\documentclass[12pt]{article}
\usepackage{color}
\usepackage{graphicx}
\usepackage{amsmath}
\usepackage{amssymb}
\usepackage{mathrsfs}
\usepackage[numbers,sort&compress]{natbib}
\usepackage{amsthm}
\numberwithin{equation}{section}
\usepackage{pgf}
\usepackage{CJK}
\usepackage{graphicx}
\usepackage{tikz}
\usepackage{stmaryrd}
\usepackage{graphicx}
\usepackage{hyperref}
\usepackage[latin1]{inputenc}
\usepackage[T1]{fontenc}
\usepackage[english]{babel}
\usepackage{listings}
\usepackage{xcolor,mathrsfs,url}
\usepackage{amssymb}
\usepackage{amsmath}
\usepackage{ifthen}
\newtheorem{theorem}{Theorem}
\newtheorem{lemma}{Lemma}
\newtheorem{corollary}{Corollary}
\newtheorem{Proposition}{Proposition}
\newtheorem{RHP}{RHP}
\newtheorem{Assumption}{Assumption}
\usepackage {caption}
\usepackage{float}
\usepackage{subfigure}

\DeclareMathOperator*{\res}{Res}

\begin{document}

\title{ On the asymptotic stability  of  $N$-soliton solutions of the three-wave resonant interaction   equation   }
\author{Yiling YANG$^1$\thanks{\ Email address: 19110180006@fudan.edu.cn } and Engui FAN$^{1}$\thanks{\ Corresponding author and email address: faneg@fudan.edu.cn } }
\footnotetext[1]{ \  School of Mathematical Sciences  and Key Laboratory   for Nonlinear Science, Fudan University, Shanghai 200433, P.R. China.}

\date{ }

\maketitle
\begin{abstract}
\baselineskip=17pt
The   three-wave resonant interaction   (three-wave) equation not only possesses $3\times 3$  matrix spectral problem,
but also being  absence  of stationary  phase points, which give rise to difficulty on the asymptotic analysis  with  stationary phase method or
 classical Deift-Zhou steepest descent method.
   In this paper, we   study the long time  asymptotics and asymptotic stability  of  $N$-soliton solutions of   the initial value problem
for the   three-wave   equation in the solitonic region
\begin{align}
	&p_{ij,t}-n_{ij}p_{ij,x}+\sum_{k=1}^{3}(n_{kj}-n_{ik})p_{ik}p_{kj}=0, \\
	&p_{ij}(x, 0)=p_{ij,0}(x), \quad x \in \mathbb{R},\  t>0,\  i,j,k=1,2,3, \nonumber\\
	&for\ i\neq j,\ p_{ij}=-\bar{p}_{ji}, \ n_{ij}=-n_{ji},
\end{align}	
where $n_{ij}$ are constants.   The study
makes crucial use of the inverse scattering transform   as well as of the $\overline\partial$ generalization of Deift-Zhou steepest descent method
 for oscillatory Riemann-Hilbert (RH) problems.
 Based on the spectral analysis of the  Lax pair associated with the   three-wave   equation and scattering matrix,
  the solution of the  Cauchy problem  is  characterized   via the solution of a  RH problem.
Further   we derive the
 leading order approximation to the  solution $p_{ij}(x, t)$ for the three-wave equation in the solitonic region of  any fixed space-time cone.
The asymptotic expansion   can be characterized with  an $N(I)$-soliton whose parameters are modulated by
a sum of localized soliton-soliton
 interactions as one moves through the region; the  residual error order
 $\mathcal{O}(t^{-1})$ from a $\overline\partial$ equation.
Our results provide  a verification of  the soliton resolution conjecture and asymptotic  stability of N-soliton solutions   for three-wave equation. 
{\bf AMS:} 35Q51; 35Q15; 37K15; 35C20.

\end{abstract}

\baselineskip=17pt

\tableofcontents

\section {Introduction}
\quad

In this paper, we  consider  the initial value problem  for the three-wave resonant interaction (three-wave) equation
\begin{align}
	&p_{ij,t}-n_{ij}p_{ij,x}+\sum_{k=1}^{3}(n_{kj}-n_{ik})p_{ik}p_{kj}=0, \label{3w}\\
	&p_{ij}(x, 0)=p_{ij,0}(x), \quad x \in \mathbb{R},\  t>0,\  i,j,k=1,2,3,\nonumber\\
	&for\ i\neq j,\ p_{ij}=-\bar{p}_{ji}, \ n_{ij}=-n_{ji}
\end{align}	
where the wave speeds $n_{ij}$ are a positive constant, and   $p_{ij}=p_{ij}(x,t)$  are a complex-valued functions of   $x$ and  $t$. Without loss of generality we will assume that
\begin{align*}
	n_{23}>n_{13}>n_{12}.
\end{align*}
The presence of resonant triads, in which two wave modes conspire to generate a third mode that grows
until the small-amplitude assumption is violated, is a primary obstacle to the effective description of small-amplitude dispersive waves by linear theory. In \cite{BUCK}, they derive a weakly nonlinear model for this process, that is the complex amplitudes of these modes satisfying the three-wave
resonant interaction  equations. The three-wave equations have a wide variety of physical
applications, originating from the fact that resonant wave coupling is such a basic nonlinear phenomenon, such as buckling of cylindrical shells \cite{33}, capillary-gravity waves \cite{42},  waves in plasmas \cite{49,51} and Rossby waves \cite{55}. 
It also has  a variety of applications to nonlinear optics like information storage and
processing \cite{2}, resonant Bragg reflection \cite{39}. So (\ref{3w}) has been found numerous applications in physics and attracted the attention of  scientific community over the last few decades. The three-wave  equation  can be solved through the inverse scattering  method because it  admits a Lax representation \cite{lax1,lax2}. This integrability give us mathematical tools to investigate several problems such as  numerically
solving the direct spectral problem for both vanishing and non vanishing boundary values \cite{D2011}, the initial-boundary value problem \cite{xu}, semiclassical soliton ensembles \cite{BUCK},
some explicit solutions \cite{Du1,Du2}, algebro-geometric quasi-periodic solutions \cite{HG}, qualitative results \cite{30},  finite-dimensional integrable system \cite{WG},   series approach (avoiding the inverse-scattering machinery) \cite{41}. Moreover, it was  shown that if the signs of $ \sum_{k=1}^{3}(n_{kj}-n_{ik})\neq0$ and do not all have the same sign, then the Cauchy problem for (\ref{3w}) with initial value $p_{ij}(x,0)\in L^{2,s}(\mathbb{R})$ for all $s>0$ has a unique global solution  in which each field is a $C^s$ function of time with values $L^{2,s}(\mathbb{R})$ for all $s>0$  (see Theorem 9.2.3) \cite{global}.

The inverse  scattering  transform (IST) procedure, as one of the most powerful tool to investigate
solitons of nonlinear integrable  models, was first discovered by Gardner,
Green, Kruskal and Miura \cite{Gardner1967}.   The development of the
IST formalism affects many fields of mathematics.   The
modern version of IST is based on the dressing method proposed by Zakharov
and Shabat, first in terms of the factorization of integral operators
on a line into a product of two Volterra integral operators  \cite{Zakharov1974}  and then
using the Riemann-Hilbert (RH) problem  \cite{Zakharov1979}. The most powerful version
of the dressing method incorporates the $\bar\partial$  problem formalism. The $\bar\partial$ problem
was put forward by Beals and Coifman as a generalization of
the RH problem and was applied to the study of first-order one-dimensional
spectral problems \cite{Beals1981,Beals1982}.  In general,  the initial value  problems  of integrable systems only  can be solved
  by suing IST or RH  method  in the case of  refectioness potentials.  So  a natural idea is  to study the asymptotic behavior of solutions
to  integrable systems.  The  study  on the long-time behavior of nonlinear wave equations   was first  carried out by Manakov in 1974 \cite{Manakov1974}.
 Later,  Zakharov and Manakov   gave   the first result on the  large-time asymptotic  of solutions for the  NLS equation  with  decaying initial value \cite{ZM1976} by this method.
   The inverse scattering method    also    worked  for long-time behavior of integrable systems    such as  KdV,  Landau-Lifshitz  and the reduced Maxwell-Bloch   system \cite{SPC,BRF,Foka}.
    In 1993,
    Deift and Zhou  developed a  nonlinear steepest descent method to rigorously obtain the long-time asymptotics behavior of the solution for the MKdV equation
by deforming contours to reduce the original RH problem   to a  model one  whose solution is calculated in terms of parabolic cylinder functions \cite{RN6}.
Since then    this method
has been widely  applied  to  the focusing NLS equation, KdV equation, Camassa-Holm equation, Degasperis-Procesi,  Fokas-Lenells equation, Sasa-Satuma equation,  short-pulse equation   etc. \cite{RN9,RN10,Grunert2009,MonvelCH,Monvel1,Monvel2,xu2015,xusp,Geng3,XF2020}.

In recent years,   McLaughlin and   Miller further extended Deift-Zhou steepest descent method  to
 a $\bar\partial$ steepest descent method,  which combine   steepest descent  with  $\bar{\partial}$-problem  rather than the asymptotic analysis
 of singular integrals on contours to analyze asymptotic of orthogonal polynomials with non-analytical weights  \cite{MandM2006,MandM2008}.
When  it  is applied  to integrable systems,   the $\bar\partial$ steepest descent method  also has  displayed some advantages,  such as   avoiding delicate estimates involving $L^p$ estimates  of Cauchy projection operators, and leading  the non-analyticity in the RH problem  reductions to a $\bar{\partial}$-problem in some sectors of the complex plane. And its result can accommodate many situations at once. In particular by considering small cones instead of fixed frames which of reference it is able to account for uncertainties in the computation (or measurement) of the spectral data and thus speed of the resulting solitons. Moreover, for focusing NLS equation, this description of Long-time asymptotic behavior to solution should also be useful to study non-integrable perturbations where the discrete spectra would no longer be stationary.
  Dieng and  McLaughin used it to study the defocusing NLS equation  under essentially minimal regularity assumptions on finite mass initial data \cite{DandMNLS}; This   method   was also successfully applied to prove asymptotic stability of N-soliton solutions to focusing NLS equation \cite{fNLS}; Jenkins et.al  studied  soliton resolution for the derivative nonlinear NLS equation  for generic initial data in a weighted Sobolev space \cite{Liu3}.  For finite density initial data, Cussagna and  Jenkins improved $\bar\partial$ steepest descent method to   study  the asymptotic stability for   defocusing NLS equation with non-zero boundary conditions \cite{SandRNLS}.  Recently   $\bar\partial$ steepest descent method has been successfully used  to
 study  the  short pulse, modifed Camassa-Holm and  Fokas-Lenells equations \cite{YF1,YF3,YF4}.

When  various  methods  above  are used to  the nonlinear evolution equations related to the higher order
matrix spectral problems, the analysis process becomes very difficult for both construction of exact solutions and  asymptotic analysis of solutions.
Up to now, the   RH methods have been  extended to construct exact  solutions for integrable nonlinear evolution equations associated with the
$3\times 3$  matrix spectral problem, such as Sasa-Satuma equation, Degasperis-Procesi,  good Boussinesq, bad Boussinesq,  three-wave,
Novikov equations \cite{Deift1982,Constantin1,Monvel3,Lenells1,Lenells2}.  However, among the these integrable systems, only the
 Degasperis-Procesi equation, coupled nonlinear Schroinger equation, Sasa-Satuma equation, cmKdV equation  have been studied
for long-time asymptotic properties \cite{Monvel1, Geng1, Geng2,Geng3}.

Compared with other integrable systems, the three-wave equation exhibits  some different characteristics,  for example, it
 not only  possesses $3\times 3$  matrix spectral problem,  but also  involves three phase functions in its corresponding  RH problem.
 However, these  three phase functions are   absence  of stationary  phase points.   To the best of
our knowledge, there is not any   result  on asymptotics for the  three-wave  equation by Deift and Zhou method or $\bar\partial$ steepest descent method.
In this paper,  we study the apply $\overline\partial$ steepest descent method  to study the asymptotic stability  of  $N$-soliton solutions of   the initial value problem
for the three-wave equation (\ref{3w}).  This  result is also a verification of the  soliton resolution conjecture    for  the three-wave  equation.

This  paper is arranged as follows.  To make our presentation easy to understand and self-contained,
 we recall  some main  results on  the construction  process  of  RH  problem with respect to  the initial problem of the three-wave equation  (\ref{3w})
  in  section \ref{sec2} ( for   example, see \cite{D2011,xu} in details),  which will be used
 to analyze   long-time asymptotics  of the three-wave equation in our paper. In section \ref{secr}, we establish the scattering
maps  from initial data  $p_{ij,0}(x)\in  H^{1,2}(\mathbb{R})$ to  the reflection coefficient $r_j(z)\in H^{1,1}(\mathbb{R})$.   In section \ref{sec3},
the function $T(z)$ is introduced to  define a new   RH problem  for  $M^{(1)}(z)$,  which  admits a regular discrete spectrum and  two  triangular  decompositions of the jump matrix
near original point.
In section \ref{sec4},  a   mixed $\bar{\partial}$-RH problem  for  $M^{(2)}(z)$ is obtained by continuous extension to $M^{(1)}(z)$  via introducing a matrix-valued  function  $R^{(2)}(z)$.
We further decompose  $M^{(2)}(z)$    into a
 model RH   problem  for  $M^{rhp}(z)$ and a  pure $\bar{\partial}$ Problem for  $M^{(3)}(z)$.
 The  $M^{rhp}(z)$  can be obtained  via  a  modified reflectionless RH problem $M^{sol}(z)$   for the soliton components which  is solved   in Section \ref{sec6}.
  In section \ref{sec7},   the error function  $E(z)$ between $M^{rhp}(z)$ and $M^{sol}(z)$ can be computed  with a  small-norm RH  problem.
 In Section \ref{sec8},   we analyze  the $\bar{\partial}$-problem  for $M^{(3)}$.
   Finally, in Section \ref{sec9},   based on  the result obtained above,   a relation formula
   is found
\begin{align}
 M(z) = M^{(3)}(z)E(z)M^{sol}(z)R^{(2)}(z)^{-1}T(z)^{-\sigma_3},\nonumber
\end{align}
from which   we then obtain the   long-time   asymptotic behavior and asymptotic stability   for the three-wave  equation (\ref{3w}) via reconstruction formula.

\section {The spectral analysis and the   RH problem}\label{sec2}

\quad At the beginning of this   section, we fix some notations used this paper.
If $I$ is an interval on the real line $\mathbb{R}$, and $X$ is a  Banach space, then $C^0(I,X)$ denotes the space of continuous functions on $I$ taking values in $X$. It is equipped with the norm
\begin{equation*}
	\|f\|_{C^{0}(I, X)}=\sup _{x \in I}\|f(x)\|_{X}.
\end{equation*}
Moreover, denote $C^0_B(X)$ as a  space of bounded continuous functions on $X$.

If the  elements  $f_1$ and $f_2$  are in space $X$,  then we call vector  $\vec{f}=(f_1,f_2)^T$  is in space $X$ with $\parallel \vec{f}\parallel_X\triangleq \parallel f_1\parallel_X+\parallel f_2\parallel_X$. Similarly, if every  entries of  matrix $A$ are in space $X$, then we call $A$ is also in space $X$.

We introduce  the following  normed spaces:\\
The weighted $L^p(\mathbb{R})$ space is defined by
$$L^{p,s}(\mathbb{R})  =  \left\lbrace f(x)\in L^p(\mathbb{R}) | \hspace{0.1cm} |x|^sf(x)\in L^p(\mathbb{R}) \right\rbrace;$$
The Sobolev space is defined by
 $$W^{k,p}(\mathbb{R})  =  \left\lbrace f(x)\in L^p(\mathbb{R}) | \hspace{0.1cm} \partial^j f(x)\in L^p(\mathbb{R})  \text{ for }j=0,1,...,k \right\rbrace;$$
The weighted Sobolev space
is defined by
$$H^{k,s}(\mathbb{R})   = \left\lbrace f(x)\in L^2(\mathbb{R}) | \hspace{0.1cm} (1+|x|^s)f(x),\partial^jf\in L^2(\mathbb{R}),  \text{ for }j=1,...,k \right\rbrace.$$
And the norm of $f(x)\in L^{p}(\mathbb{R})$ and $g(x)\in L^{p,s}(\mathbb{R})$ are  abbreviated to $\parallel f\parallel_{p}$,  $\parallel g\parallel_{p,s}$ respectively.

\quad The three-wave equation (\ref{3w})  admits the Lax pair \cite{lax1,lax2}
\begin{equation}
\Phi_x =\left(  izA+P\right)  \Phi,\hspace{0.5cm}\Phi_t =\left(izB+Q \right)  \Phi, \label{lax0}
\end{equation}
while $\Phi(x,t,z)$ is a common 3-dim vector solution. $A$ and $  B$ are   real   diagonal constant  matrices given by
$$A={\rm diag} \{a_1,a_2,a_3\}, \ \ B= {\rm diag} \{b_1,b_2,b_3\}$$
satisfying
 tr($A$)=tr($B$)=0. $P(x,t)$, $Q(x,t)$ are matrix valued functions given by
\begin{equation}
	P=\left(\begin{array}{ccc}
		0 & p_{12} & p_{13} \\
		-\bar{p}_{12} & 0 & p_{23} \\
		-\bar{p}_{13} & -\bar{p}_{23} & 0
	\end{array}\right),\ Q=\left(\begin{array}{ccc}
	0 & n_{12}p_{12} & n_{13}p_{13} \\
	n_{12}\bar{p}_{12} & 0 & n_{23}p_{23} \\
	n_{13}\bar{p}_{13} & n_{23}\bar{p}_{23} & 0
\end{array}\right),\nonumber
\end{equation}
where $n_{ij}=\frac{b_i-b_j}{a_i-a_j}$.  Besides, since $a_1, a_2, a_3$ and $n_{23},n_{13},n_{12}$ are real,  we assume that $a_1>a_2>a_3, \ n_{23}>n_{13}>n_{12}$  without loss of generality.
We  first recall  some main  results on  the construction  process  of  RH  problem.
Making   transformation
\begin{equation}
	\Phi_\pm=\mu_\pm e^{iz(xA+tB)}\label{transmu},
\end{equation}
then
$$\mu_\pm \sim I, \hspace{0.5cm} x \rightarrow \pm\infty,$$
and the   system (\ref{lax0}) then becomes
\begin{align}
	&(\mu_\pm)_x = iz[A,\mu_\pm]+P\mu_\pm,\label{lax1.1}\hspace{0.5cm}\\
	&(\mu_\pm)_t = iz[B,\mu_\pm]+Q\mu_\pm,\label{lax1.2}\hspace{0.5cm}
\end{align}
which leads to two  Volterra type integrals
\begin{equation}
	\mu_\pm=I+\int_{\pm \infty}^{x}e^{iz\widehat{A}(x-y)}P(y)\mu_\pm(y)dy\label{intmu}.
\end{equation}
The Able  formula gives that $\det (\mu_\pm)=\det (\Phi_\pm)=1$. Denote
$$\mu_\pm=(\mu_{\pm,ij})_{3\times3}=\left(\left[ \mu_\pm\right]_1, \left[ \mu_\pm\right]_2,\left[ \mu_\pm\right]_3 \right), $$
where  $\left[ \mu_\pm\right] _i$ for $i=1,2,3$  are
the $i$-th  columns of $\mu_\pm$ respectively.
Then  from  (\ref{intmu}),   we can show that  $\left[ \mu_-\right] _3$ and $\left[ \mu_+\right] _1$ are analytical  in $\mathbb{C}^+$;  $\left[ \mu_+\right] _3$
and $\left[ \mu_-\right] _1$ are analytical  in $\mathbb{C}^-$.
Denote $X^A$ is the cofactor matrix of a $3\times3$ matrix $X$. It follows from (\ref{lax1.1}) that the conjugate eigenfunction $\mu^A$ satisfies the Lax pair:
\begin{align}
&(\mu^A_\pm)_x = -iz[A,\mu_\pm^A]-P^T\mu_\pm^A,\\
&(\mu^A_\pm)_t = -iz[B,\mu_\pm^A]-Q^T\mu_\pm^A,
\end{align}
which leads to two  Volterra type integrals:
\begin{equation}
	\mu^A_\pm=I-\int_{\pm \infty}^{x}e^{-iz\widehat{A}(x-y)}P^T(y)\mu^A_\pm(y)dy\label{intmuA}.
\end{equation}
 $\bar{P}^T=-P$ and $\bar{Q}^T=-Q$ imply the  following symmetry:
\begin{align}
	\mu_\pm(z)=\overline{\mu^A_\pm(\bar{z})}.\label{sym}
\end{align}
Since   $\Phi_\pm(z;x,t)$ are two fundamental matrix solutions of the  Lax  pair (\ref{lax0}),  there exists a linear  relation between $\Phi_+(z;x,t)$ and $\Phi_-(z;x,t)$, namely,
\begin{align}
&\Phi_-(z;x,t)=\Phi_+(z;x,t)S(z),\hspace{0.5cm} z\in \mathbb{R},\\
& S(z) =(s_{ij}(z))_{3\times3},\hspace{0.5cm}\det S(z)=1,\label{scattering}
\end{align}
where $S(z)$ is called scattering matrix and only depends on $z$. And combing with (\ref{transmu}), above equation is changed into
\begin{align}
	\mu_+(z)=\mu_-(z)e^{iz(x\widehat{A}+t\widehat{B})}S(z).\label{s}
\end{align}
Consider the cofactor matrix of $S(z)$
 \begin{align}
 	\mu^A_+(z)=\mu^A_-(z)e^{-iz(x\widehat{A}+t\widehat{B})}S^A(z).\label{sA}
 \end{align}
Then $S(z)$ and $S^A(z)$ admit symmetry reduction as
\begin{equation}
	S(z)=\overline{S^A(\bar{z})}.\label{symS}
\end{equation}
Moreover, $s_{11}(z)$, $s^A_{33}(z)$ are analysis in $\mathbb{C}^+$, while $s_{33}(z)$, $s^A_{11}(z)$ are analysis in $\mathbb{C}^-$ with $s_{11}(z)=\overline{s^A_{11}(\bar{z})}$ and $s_{33}(z)=\overline{s^A_{33}(\bar{z})}$. The reflection coefficients are defined by
\begin{equation}
	r_1(z)=\frac{s_{12}(z)}{s_{11}(z)},\ r_2(z)=\frac{s_{31}(z)}{s_{33}(z)},\ r_3(z)=\frac{s_{32}(z)}{s_{33}(z)},\ r_4(z)=\frac{s_{13}(z)}{s_{11}(z)},\label{r}
\end{equation}
with $r_4+r_1\bar{r}_3+\bar{r}_2=0$. In addition,   $\mu_\pm(z)$  admits the  asymptotics
\begin{align}
	\mu_\pm(z)=I+\mathcal{O}(z^{-1}),\hspace{0.5cm}z \rightarrow \infty,\label{asymu}
\end{align}
with reconstruction formula
\begin{equation}
	p_{ij}=-i(a_i-a_j)\lim_{z\to \infty}[z\mu(z)]_{ij}.
\end{equation}
And the scattering matrix satisfy
\begin{align}
	S(z)=I+\mathcal{O}(z^{-1}),\ S^A(z)=I+\mathcal{O}(z^{-1}),\hspace{0.5cm}z \rightarrow \infty.
\end{align}

The zeros of $s_{11}(z)$ and $s^A_{33}(z)$ on $\mathbb{R}$   are known to
occur and they correspond to spectral singularities.  They are excluded from our analysis in the this paper.  Recall the main result in \cite{BC} by  Beals and Coifman:
\begin{lemma}
	There exists a dense open set $P_0\subset L^1(\mathbb{R})$ such that if $p_{ij,0}(x) \in   P_0$, then
	$s_{11}(z)$ and $s^A_{33}(z)$  only has finite number of simple zeros.
\end{lemma}
In next section \ref{secr}, we establish the relationship between initial data  $p_{ij,0}(x)\in  H^{1,2}(\mathbb{R})$ to
 the reflection coefficient $r_j(z)\in H^{1,1}(\mathbb{R})$.
 To deal with our following work,
we assume our initial data satisfy this assumption.
\begin{Assumption}\label{initialdata}
	The initial data $p_{ij,0}(x) \in  H^{1,2}(\mathbb{R})\cap P_0$     and it generates generic scattering data which satisfy that
 $s_{11}(z)$ and $s^A_{33}(z)$  has no zeros on $\mathbb{R}$.
\end{Assumption}
In fact, since scattering data  $s_{11}(z),\  s^A_{33}(z)$ are analytical in $\mathbb{C}^+$ and $ s_{11}(z), s^A_{33}(z) \rightarrow 1, \ z\rightarrow \infty$,
we can deduce that $s_{11}(z), s^A_{33}(z)$ have finite zeros in $\mathbb{C}^+$.   And suppose that $s_{11}(z)$ has $N_1$ simple zeros $z_1,...,z_{N_1}$ on $\mathbb{C}^+$, and $s^A_{33}(z)$ has $N_2$ simple
zeros $z_{N_1+1},...,z_{N_1+N_2}$ on $\mathbb{C}^+$.   The  symmetries  (\ref{symS}) imply that $\bar{z}_1,...,\bar{z}_{N_1}$ and $\bar{z}_{N_1+1},...,\bar{z}_{N_1+N_2}$ are the simple zeros of $s^A_{11}(z)$ and $s_{33}(z)$ respectively. Denote the discrete spectrum as
\begin{equation}
	\mathcal{Z}=\left\{ z_n, \  \bar{z}_n\right\}_{n=1}^{N_1+N_2}. \label{spectrals}
\end{equation}
The distribution  of $	\mathcal{Z}$ on the $z$-plane   is shown  in Figure \ref{fig:figure1}.
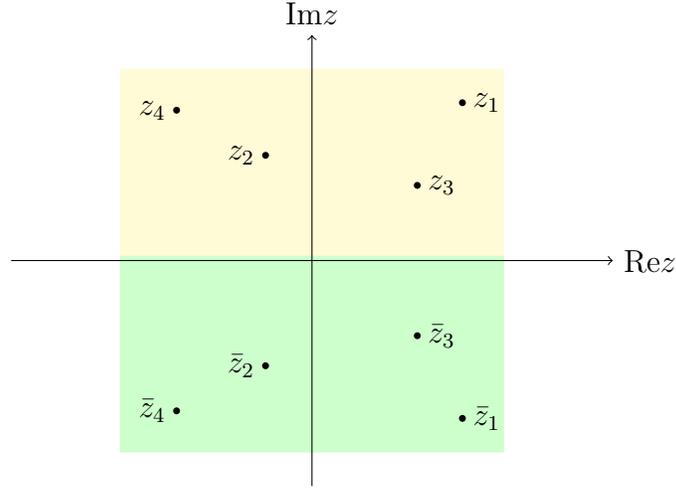
\begin{figure}[H]
	\centering
	\begin{tikzpicture}[node distance=2cm]
		\filldraw[yellow!20,line width=3] (2.5,0.01) rectangle (0.01,2.5);
		\filldraw[yellow!20,line width=3] (-2.5,0.01) rectangle (-0.01,2.5);
\filldraw[green!20,line width=3] (2.5,0.01) rectangle (0.01,-2.5);
		\filldraw[green!20,line width=3] (-2.5,0.01) rectangle (-0.01,-2.5);
		\draw[->](-4,0)--(4,0)node[right]{Re$z$};
		\draw[->](0,-3)--(0,3)node[above]{Im$z$};
	\coordinate (A) at (2,2.1);
\coordinate (B) at (2,-2.1);
\coordinate (C) at (-0.616996232,1.4);
\coordinate (D) at (-0.616996232,-1.4);
		\coordinate (E) at (1.4,1);
		\coordinate (F) at (1.4,-1);
		\coordinate (G) at (-1.8,2);
		\coordinate (H) at (-1.8,-2);
		\fill (A) circle (1.3pt) node[right] {$z_1$};
		\fill (B) circle (1.3pt) node[right] {$\bar{z}_1$};
		\fill (C) circle (1.3pt) node[left] {$z_2$};
		\fill (D) circle (1.3pt) node[left] {$\bar{z}_2$};
		\fill (E) circle (1.3pt) node[right] {$z_3$};
		\fill (F) circle (1.3pt) node[right] {$\bar{z}_3$};
		\fill (G) circle (1.3pt) node[left] {$z_4$};
		\fill (H) circle (1.3pt) node[left] {$\bar{z}_4$};
	\end{tikzpicture}
	\caption{Distribution of the discrete spectrum $\mathcal{Z}$. }
	\label{fig:figure1}
\end{figure}

Define  a   sectionally meromorphic matrix
\begin{align}
	M_+(z)=\left(\begin{array}{ccc}
		\mu_{+,11} & \frac{1}{s_{11}}( 	\mu^A_{-,31}\mu^A_{+,23}-\mu^A_{-,21}\mu^A_{+,33})   & \frac{\mu_{-,13}}{s^A_{33}}\\ 	
		\mu_{+,21} & \frac{1}{s_{11}}( 	\mu^A_{-,11}\mu^A_{+,33}-\mu^A_{-,31}\mu^A_{+,13})   & \frac{\mu_{-,23}}{s^A_{33}}\\
		\mu_{+,31} & \frac{1}{s_{11}}( 	\mu^A_{-,21}\mu^A_{+,13}-\mu^A_{-,11}\mu^A_{+,23})   & \frac{\mu_{-,33}}{s^A_{33}}
	\end{array}\right),\text{as } z\in \mathbb{C}^+;\\
M_-(z)=\left(\begin{array}{ccc}
	\frac{\mu_{-,11}}{s^A_{11}} & \frac{1}{s_{33}}( 	\mu^A_{+,31}\mu^A_{-,23}-\mu^A_{+,21}\mu^A_{-,33})   & \mu_{+,13}\\ 	
	\frac{\mu_{-,12}}{s^A_{11}} & \frac{1}{s_{33}}( 	\mu^A_{+,11}\mu^A_{-,33}-\mu^A_{+,31}\mu^A_{-,13})   & \mu_{+,23}\\
	\frac{\mu_{-,13}}{s^A_{11}} & \frac{1}{s_{33}}( 	\mu^A_{+,21}\mu^A_{-,13}-\mu^A_{+,11}\mu^A_{-,23})   & \mu_{+,33}
\end{array}\right),\text{as } z\in \mathbb{C}^-,
\end{align}
with their cofactor matrix:
\begin{align}
	M_+^A(z)=\left(\begin{array}{ccc}
		\frac{\mu^A_{-,11}}{s_{11}} & \frac{1}{s^A_{33}}( 	\mu_{+,31}\mu_{-,23}-\mu_{+,21}\mu_{-,33})   & \mu^A_{+,13}\\ 	
		\frac{\mu^A_{-,12}}{s_{11}} & \frac{1}{s^A_{33}}( 	\mu_{+,11}\mu_{-,33}-\mu_{+,31}\mu_{-,13})   & \mu^A_{+,23}\\
		\frac{\mu^A_{-,13}}{s_{11}} & \frac{1}{s^A_{33}}( 	\mu_{+,21}\mu_{-,13}-\mu_{+,11}\mu_{-,23})   & \mu^A_{+,33}
	\end{array}\right),\text{as } z\in \mathbb{C}^+;\\
	M_-^A(z)=\left(\begin{array}{ccc}
	\mu^A_{+,11} & \frac{1}{s^A_{11}}( 	\mu_{-,31}\mu_{+,23}-\mu_{-,21}\mu_{+,33})   & \frac{\mu^A_{-,13}}{s_{33}}\\ 	
	\mu^A_{+,21} & \frac{1}{s^A_{11}}( 	\mu_{-,11}\mu_{+,33}-\mu_{-,31}\mu_{+,13})   & \frac{\mu^A_{-,23}}{s_{33}}\\
	\mu^A_{+,31} & \frac{1}{s^A_{11}}( 	\mu_{-,21}\mu_{+,13}-\mu_{-,11}\mu_{+,23})   & \frac{\mu^A_{-,33}}{s_{33}}
\end{array}\right),\text{as } z\in \mathbb{C}^-.
\end{align}
\begin{Proposition}
	 $M_\pm(z)$ can be construct in another way by $\mu_\pm(z)$ as
	 \begin{align}
	 	M_+(z)=\mu_+e^{iz(x\widehat{A}+t\widehat{B})}\left(\begin{array}{ccc}
	 		1 & -\frac{s_{12}}{s_{11}} & \frac{s^A_{31}}{s^A_{33}}\\ 	
	 		0 & 1   & \frac{s^A_{32}}{s^A_{33}}\\
	 		0 & 0   & 1
	 		\end{array}\right)=\mu_-e^{iz(x\widehat{A}+t\widehat{B})}\left(\begin{array}{ccc}
	 		s_{11} & 0 & 0\\ 	
	 		s_{21} & \frac{s^A_{33}}{s_{11}}   & 0\\
	 		s_{31} & \frac{s^A_{23}}{s_{11}}   & \frac{1}{s^A_{33}}
 		\end{array}\right),\\
 		M_-(z)=\mu_+e^{iz(x\widehat{A}+t\widehat{B})}\left(\begin{array}{ccc}
 			1 & 0 & 0\\ 	
 			\frac{s^A_{12}}{s^A_{11}} & 1  &0 \\
 			\frac{s^A_{13}}{s^A_{11}} & -\frac{s_{32}}{s_{33}}   & 1
 		\end{array}\right)=\mu_-e^{iz(x\widehat{A}+t\widehat{B})}\left(\begin{array}{ccc}
 		\frac{1}{s^A_{11}} & \frac{s^A_{21}}{s_{33}} & s_{13}\\ 	
 			0 & \frac{s^A_{11}}{s_{33}}   & s_{23}\\
 			0 & 0   & s_{33}
 		\end{array}\right).
	 \end{align}
\end{Proposition}

  We determine the residue conditions at these zeros. Denote   norming constants
 	 \begin{align}
 &  c_n=-\frac{s_{12}(z_n)}{s_{11}'(z_n)}, \ {\rm for}\   n=1,...,N_1; \ \ c_n=\frac{s^A_{23}(z_n)}{s'^A_{33}(z_n)s_{11}(z_n)}, \ {\rm for}\  n=N_1+1,...,N_1+N_2,\nonumber \\
& \tilde{c}_n=\frac{s_{33}(\bar{z}_n)}{s'^A_{11}(\bar{z}_n)s^A_{21}(\bar{z}_n)}, \ {\rm for}\   n=1,...,N_1; \ \ \tilde{c}_n=-\frac{s_{32}(\bar{z}_n)}{s_{33}'(\bar{z}_n)}, \ {\rm for}\  n=N_1+1,...,N_1+N_2.\nonumber
	 \end{align}
  And the collection $\sigma_d=  \left\lbrace z_n,c_n\right\rbrace^{N_1+N_2}_{n=1}  $
is called the \emph{scattering data}. Denote the phase functions
\begin{align}
	\theta_{ij}=(a_i-a_j)\xi+(b_i-b_j)\label{theta},\hspace{0.5cm}i,j=1,2,3,\hspace{0.5cm}\xi=\frac{x}{t},
\end{align}
with $\theta_{ij}=-\theta_{ji}$
Then  we have  the   following RH problem.
\begin{RHP}\label{RHP1}
	 Find a matrix-valued function $M(z )=M(z;x,t)$ which satisfies
	
	$\blacktriangleright$ Analyticity: $M(z )$ is meromorphic in $\mathbb{C}\setminus \mathbb{R}$ and has single poles on  $ \mathcal{Z}\cup \bar{\mathcal{Z}}$;
	
	$\blacktriangleright$ Symmetry: $M(z)=\overline{M^A(\bar{z})}$;
	
	$\blacktriangleright$ Jump condition: $M$ has continuous boundary values $M_\pm$ on $\mathbb{R}$ and
	\begin{equation}
		M^+(z )=M^-(z )V(z),\hspace{0.5cm}z \in \mathbb{R},
	\end{equation}
	where
	\begin{equation}
		V(z)=e^{iz(x\widehat{A}+t\widehat{B})}\left(\begin{array}{ccc}
			1 & -r_1 & \bar{r}_2\\ 	
			-\bar{r}_1 & 1+|r_1|^2 & \bar{r}_3-\bar{r}_1\bar{r}_2 \\
			r_2 & r_3-r_1r_2   & 1+|r_3|^2+|r_2|^2
		\end{array}\right)\label{jumpv};
	\end{equation}
	
	$\blacktriangleright$ Asymptotic behaviors:
	\begin{align}
		&M(z) = I+\mathcal{O}(z^{-1}),\hspace{0.5cm}z \rightarrow \infty;
	\end{align}
	
	$\blacktriangleright$ Residue conditions: $M(z)$ has simple poles at each point in $ \mathcal{Z}\cup \bar{\mathcal{Z}}$ with

	\quad For $n=1,...,N_1$:
	\begin{align}
		&\res_{z=z_n}M(z)=\lim_{z\to z_n}M(z)\left(\begin{array}{ccc}
			0 & c_ne^{iz_nt\theta_{12}} & 0\\ 	
			0 & 0 & 0 \\
			0 & 0   & 0
		\end{array}\right)\label{RES1},\\
		&\res_{z=\bar{z}_n}M(z)=\lim_{z\to \bar{z}_n}M(z)\left(\begin{array}{ccc}
			0 & 0 & 0\\ 	
			\tilde{c}_ne^{-i\bar{z}_nt\theta_{12}} & 0 & 0 \\
			0 & 0   & 0
		\end{array}\right);\label{RES2}
	\end{align}

\quad For $n=N_1+1,...,N_1+N_2$,
	\begin{align}
	&\res_{z=z_n}M(z)=\lim_{z\to z_n}M(z)\left(\begin{array}{ccc}
		0 & 0 & 0\\ 	
		0 & 0 & c_ne^{iz_nt\theta_{23}} \\
		0 & 0   & 0
	\end{array}\right)\label{RES3},\\
	&\res_{z=\bar{z}_n}M(z)=\lim_{z\to \bar{z}_n}M(z)\left(\begin{array}{ccc}
		0 & 0 & 0\\ 	
		0 & 0 & 0 \\
		0 & \tilde{c}_ne^{-i\bar{z}_nt\theta_{23}}   & 0
	\end{array}\right)\label{RES4}.
\end{align}
\end{RHP}

From the asymptotic behavior of the functions $\mu_\pm(z)$, $S(z)$ and their cofactor matrix, we have following reconstruction formula
\begin{equation}
	p_{ij}(x,t)=-i(a_i-a_j)\lim_{z\to \infty}[zM(z)]_{ij}, \ \ i,j=1,2,3.\label{recons u}
\end{equation}

\section{The scattering maps}\label{secr}

\quad
We   consider the  $x$-part of the  Lax pair (\ref{lax0}) to analyze the initial value problem and give the proof of proposition \ref{pror} in this section.   In fact,  taking account of $t$-part of  the Lax pair and though the standard direct scattering transform method, then it deduce that $S(z)$ have  time evolution:  $S(z,t)=e^{iz\widehat{B}t}S(z,0)$. So $|r_i(z,t)|=|r_i(z,0)|$ for $i=1,2,3,4$.
In this section, we abbreviate  $\mu_\pm(x,0,z)$, $\mu^A_\pm(x,0,z)$ as $\mu_\pm(x,z)$, $\mu^A_\pm(x,z)$ respectively.

 To find  the relationship of initial value and reflection coefficients, we recall  two Volterra integral equations in (\ref{intmu}) and (\ref{intmuA})
\begin{align}
	&\mu_\pm(x)=I+\int_{\pm \infty}^{x}e^{iz\widehat{A}(x-y)}P(y)\mu_\pm(y)dy,\label{lax3.1}\\
	&\mu^A_\pm(x)=I-\int_{\pm \infty}^{x}e^{-iz\widehat{A}(x-y)}P^T(y)\mu^A_\pm(y)dy.
\end{align}
We need estimates on the $L^2$-integral property of $\mu_\pm(z)$ $\mu^A_\pm(z)$ and their derivatives.  And we abbreviate
$C_B^0(\mathbb{R}_x^\pm\times\mathbb{R}_z )$, $ C^0(\mathbb{R}_x^\pm,L^2(\mathbb{R}_z))$, $ L^2(\mathbb{R}_x^\pm\times \mathbb{R}_z)$ to $C_B^0$, $ C^0$, $ L^2_{xz}$ respectively.
The following  result  is  useful in the analysis of  direct scattering map,  namely estimates for Volterra-type integral equations above.
\begin{lemma}\label{lemma1}
	Suppose that  $F$ is a three factorial square matrix and $g$ is a  column  vector, then
$$|Fg|\leq|F||g|.$$
\end{lemma}
\begin{lemma}\label{lemma2}
	For $\psi(\eta)\in L^2(\mathbb{R})$, $f(x)\in L^{2,1/2}(\mathbb{R})$, following inequality hold:
	\begin{align}
&\bigg|\int_{\mathbb{R}}\int_{x}^{\pm \infty}f(y)e^{-\frac{i}{2}\eta(p(x)-p(y))}\psi(\eta)dyd\eta \bigg| \nonumber\\
&=\bigg|\int_{x}^{\pm \infty}f(y)\psi(\frac{1}{2}(p(x)-p(y)))dy\bigg|\lesssim \left( \int_{x}^{\pm \infty}|f(y)|^2dy\right)^{1/2}  \parallel \psi\parallel_2;\\
&\int_{0}^{\pm \infty}\int_{\mathbb{R}}\bigg| \int_{x}^{\pm \infty}f(y)e^{-\frac{i}{2}\eta(p(x)-p(y))} dy\bigg|^2d\eta dx  \lesssim \parallel f\parallel_{2,1/2}^2.
	\end{align}
\end{lemma}
The proof of above lemmas are trivial, so we omit it. From the symmetry reduction (\ref{sym}), we will only consider $\mu_\pm(x,z)$. And we just give the detail for the first column of $\mu_\pm(x,z)$. For the sake of brevity, we  denote
\begin{align}
f^{(j)}(x,y,z)=e^{-iz(a_1-a_j)(x-y)},  j=2,3,; \ \ [\mu_\pm]_1(x,z)-e_1\triangleq n_\pm(x,z),\label{n}
\end{align}
where $e_1$ is identity vector $(1,0,0)^T$. Introduce the integral operator
\begin{align}\label{Tpm}
	T_\pm(f)(x,z)=\int_{x}^{\pm\infty}K_\pm(x,y,z)f(y,z)dy,
\end{align}
where integral kernel $K_\pm(x,y,z)$ is
\begin{align}\label{K}
	&K_\pm(x,y,z)=\left(\begin{array}{ccc}
		0 & p_{12}  & p_{13}\\
		-f^{(2)}\bar{p}_{12} & 0 & f^{(2)}p_{23}\\
		-f^{(3)}\bar{p}_{13} & -f^{(3)}\bar{p}_{23} & 0
	\end{array}\right).
\end{align}
Then (\ref{lax3.1}) is changed into
\begin{align}\label{eqn}
	&n_\pm=T_\pm(e_1)+T_\pm(n_\pm)\triangleq n^0_\pm+T_\pm(n_\pm).	
\end{align}
Differentiating above equation  with respect to  $z$ yields
\begin{align}\label{eqnk}
	&[n_\pm]_z=n_\pm^1+T_\pm([n_\pm]_z),\hspace{0.5cm}n_\pm^1=[n^0_\pm]_z+[T_\pm]_z(n_\pm).
\end{align}
$[T_\pm]_z$ is also a integral operator with integral kernel $[K_\pm]_z(x,y,z)$:
\begin{align}
	[K_\pm]_z(x,y,z)=(x-y)\left(\begin{array}{ccc}
		0 & 0  & 0\\
		i(a_1-a_2)f^{(2)}\bar{p}_{12} & 0 & -i(a_1-a_2)f^{(2)}p_{23}\\
		i(a_1-a_3)f^{(3)}\bar{p}_{13} & i(a_1-a_3)f^{(3)}\bar{p}_{23} & 0
	\end{array}\right).\label{Kz}
\end{align}
\begin{lemma}\label{lemma4}
	For the integral operators  $T_\pm$ and $[T_\pm]_z$  defined above,
then  $n^0_\pm(x,z)=T_\pm(e_1)(x,z)$ and $[n^0_\pm]_z(x,z)$ are in $ C_B^0\cap C^0\cap L^2_{xz}$.
\end{lemma}
\begin{proof}
	$n^0_\pm(x,z)$ is given by
	\begin{align}
		n^0_\pm(x,z)=T_\pm(e_1)(x,z)=\int_{x}^{\pm\infty}\left(\begin{array}{ccc}
			0  \\
			-f^{(2)}(x,y,z)\bar{p}_{12}(y)\\
			-f^{(3)}(x,y,z)\bar{p}_{13}(y)
		\end{array}\right)dy,\label{Te1}
	\end{align}
	with
	\begin{align}
		|n^0_\pm(x,z)|&\leq \bigg|\int_{x}^{\pm\infty}f^{(2)}(x,y,z)\bar{p}_{12}(y)dy\bigg|+\bigg|\int_{x}^{\pm\infty}f^{(3)}(x,y,z)\bar{p}_{13}(y)dy\bigg|\\
		&\leq \parallel p_{12}(x) \parallel_1+\parallel p_{13}(x) \parallel_1\leq\parallel p_{12} (x) \parallel_{2,1}+\parallel p_{13} (x)\parallel_{2,1}.
	\end{align}
	Then from $p_{ij}(x)\in H^{1,2}$, by lemma \ref{lemma2}, it   immediately derives to
	\begin{align}
		&\parallel n^0_\pm \parallel_{C^0}\lesssim\parallel p_{12} (x)\parallel_2+\parallel p_{13} (x)\parallel_2,\\
		&\parallel n^0_\pm \parallel_{L^2_{xz}}\lesssim\parallel p_{12}(x) \parallel_{2,1/2}+\parallel p_{13}(x) \parallel_{2,1/2}.
	\end{align}
	And  $[n^0_\pm]_z(x,z)(x,z)$ is  given by
	\begin{align}
		[n^0_\pm]_z(x,z)=[T_\pm]_z(e_1)(x,z)=&\int_{x}^{\pm\infty}\left(\begin{array}{ccc}
			0  \\
			i(a_1-a_2)(x-y)f^{(2)}(x,y,z)\bar{p}_{12}(y)\\
			i(a_1-a_3)(x-y)f^{(3)}(x,y,z)\bar{p}_{13}(y)
		\end{array}\right)dy.
	\end{align}
	Similarly from  lemma \ref{lemma2}, it achieves that
	\begin{align}
		&|[n^0_\pm]_z|\leq \parallel p_{12}(x) \parallel_{1,1}+\parallel p_{13} (x)\parallel_{1,1}\leq\parallel p_{12}(x) \parallel_{2,2}+\parallel p_{13}(x) \parallel_{2,2}\\
		&\parallel [n^0_\pm]_z \parallel_{C^0}\lesssim \parallel p_{12}(x) \parallel_{2,1}+\parallel p_{13}(x) \parallel_{2,1},\\
		&\parallel [n^0_\pm]_z \parallel_{L^2_{xz}}\lesssim \parallel p_{12} (x)\parallel_{2,3/2}+\parallel p_{13} (x)\parallel_{2,3/2}.
	\end{align}
\end{proof}

The operator $T_\pm$  and $[T_\pm]_z$ induce linear mappings  given in next lemma.
\begin{lemma}
	The integral operator $T_\pm$ and its $z$-derivative $[T_\pm]_z$ map $C_B^0\cap C^0\cap L^2_{xz}$ to itself. Moreover, $(I-T^\pm)^{-1} $exists as a bounded	operator on  $C_B^0\cap C^0\cap L^2_{xz}$.
\end{lemma}
\begin{proof}
	In fact,  (\ref{K}) leads to
	\begin{align}
		|K_\pm(x,y,z)|=|P(y)|.
	\end{align}
	For any $f(x,z) \in C_B^0\cap C^0\cap L^2_{xz}$, by Lemma \ref{lemma1}, we have
	\begin{align}
		|T_\pm(f)(x,z)|\leq \int_{x}^{\pm\infty}|P(y)|dy \parallel f \parallel_{C_B^0}.
	\end{align}
	Moreover,
	\begin{align}
		\left( \int_{\mathbb{R}}|T_\pm(f)(x,z)|^2dz\right) ^{\frac{1}{2}}&=\left( \int_{\mathbb{R}}\bigg|\int_{x}^{\pm\infty}K_\pm(x,y,z)f(y,z)dy\bigg|^2dz\right) ^{\frac{1}{2}}\nonumber\\
		&\leq \bigg|\int_{x}^{\pm\infty}\left(\int_{\mathbb{R}} | K_\pm(x,y,z)|^2|f(y,z)|^2dz\right) ^{\frac{1}{2}}dy\bigg|\nonumber\\
		&=\bigg|\int_{x}^{\pm\infty}|P(y)|\parallel f(y,z) \parallel_{L^2_z}dy\bigg|\leq\parallel P \parallel_{1}\parallel f \parallel_{C^0},
	\end{align}
	which derives to
	\begin{align}
	\left(\int_{\mathbb{R}^\pm} \int_{\mathbb{R}}|T_\pm(f)(x,z)|^2dzdx\right) ^{\frac{1}{2}}&\leq \bigg|
	\int_{\mathbb{R}^\pm}\left( \int_{x}^{\pm\infty}|P(y)|\parallel f(y,z) \parallel_{L^2_z}dy\right) ^2dx \bigg|^{\frac{1}{2}}\nonumber\\
	&\leq\bigg|\int_{\mathbb{R}^\pm}\left( \int_{0}^y|P(y)|^2\int_{\mathbb{R}}|f(y,z)|^2dzdx\right)dy\bigg|^{\frac{1}{2}}\nonumber\\
	&\leq\parallel P \parallel_{2,1/2}\parallel f \parallel_{L^2_{xz}}.
	\end{align}
	Denote $K^n_\pm$ is the integral kernel of Volterra operator $[T_\pm]^n$ as
	\begin{align}
		K^n_\pm(x,y_n,z)=\int_{x}^{y_n}...\int_{x}^{y_2}K_\pm(x,y_1,z)K_\pm(y_1,y_2,z)...K_\pm(y_{n-1},y_n,z)dy_1...dy_{n-1},
	\end{align}
	with
	\begin{align}
		|K^n_\pm(x,y,z)|\leq \frac{1}{(n-1)!}\left( \int_{x}^{\pm\infty}|P(y)|dy \right) ^{n-1}|P(y)|.
	\end{align}
	Analogously, $[T_\pm]^n$ admits that
	\begin{align}
	&\parallel [T_\pm]^n \parallel_{\mathcal{B}(C_B^0)}\leq \frac{\parallel P \parallel_{1}^n}{(n-1)!},\ \
		\parallel [T_\pm]^n \parallel_{\mathcal{B}(C^0)}\leq \frac{\parallel P \parallel_{1}^n}{(n-1)!},\nonumber\\
&  \parallel [T_\pm]^n \parallel_{\mathcal{B}(L^2_{xz})}\leq \frac{\parallel P \parallel_{1}^{(n-1)}}{(n-1)!}\parallel P \parallel_{2,1/2}.\nonumber
	\end{align}
	Then  the standard Volterra theory gives the following operator norm:
	\begin{align}
		&\parallel (I-T_\pm)^{-1} \parallel_{\mathcal{B}(C_B^0)}\leq e^{\parallel P \parallel_{1}}\parallel P \parallel_{1},\nonumber\\
		&\parallel (I-T_\pm)^{-1} \parallel_{\mathcal{B}(C^0)}\leq e^{\parallel P \parallel_{1}}\parallel P \parallel_{1},\nonumber\\
		&\parallel (I-T_\pm)^{-1} \parallel_{\mathcal{B}(L^2_{xz})}\leq e^{\parallel P \parallel_{1}}\parallel P \parallel_{2,1/2}.\nonumber
	\end{align}
 As for the $z$-derivative $[T_\pm]_z$, from
 \begin{align}
 	|[K_\pm]_z|\lesssim |x-y||P(y)|,\nonumber
 \end{align}
it analogously leads to
\begin{align}
	\parallel [T_\pm]_z \parallel_{\mathcal{B}(C_B^0)}\leq \parallel P \parallel_{1,1},\ \parallel [T_\pm]_z \parallel_{\mathcal{B}(C^0)}\leq \parallel P \parallel_{1,1},\ \parallel [T_\pm]_z \parallel_{\mathcal{B}(L^2_{xz})}\leq \parallel P \parallel_{2,3/2}.\nonumber
\end{align}
\end{proof}
By using  above lemma,  we can show that  $[T_\pm]_z(n_\pm)\in C_B^0\cap C^0\cap L^2_{xz}$,
which implies that $n_\pm^1\in C_B^0\cap C^0\cap L^2_{xz}$.
Since the operator $(I-T_\pm)^{-1}$
exist, the equations (\ref{eqn})-(\ref{eqnk})  are  solvable with
\begin{align}
	&n_\pm(x,z)=(I-T_\pm)^{-1}(n_\pm^0)(x,z),\\
	&[n^\pm(x,z)]_z=(I-T_\pm)^{-1}(n_\pm^1)(x,z).
\end{align}
Combining above Lemmas and the definition  of $n_\pm$ (\ref{n}), we immediately obtain the following property of $\mu_\pm(x,z)$.
\begin{Proposition}\label{mu1}
	Suppose that $p_{ij0}(x)\in H^{1,2}(\mathbb{R})$ for $i,j=1,2,3$, then $\mu_\pm(0,z)-I$ and its $z$-derivative $[\mu_\pm(0,z)]_z$ belong in $C_B^0(\mathbb{R})\cap L^2(\mathbb{R})$.
\end{Proposition}
Then we begin to prove proposition \ref{pror}. we  rewrite (\ref{s}) and (\ref{sA}) at $t=0$:
\begin{align}
	&\mu_+(z)=\mu_-(z)e^{izx\widehat{A}}s(z),\label{scatteringcoefficient2}\\
	&\mu^A_+(z)=\mu^A_-(z)e^{-izx\widehat{A}}s^A(z).
\end{align}
It is requisite to shown
$r_i(z),\hspace{0.3cm}r'_i(z),\hspace{0.3cm}zr_i(z)\text{ in }L^2(\mathbb{R})$ for $i=1,2,3,4$. We only give the proof for $i=1$, the others are analogously. Denote $m_\pm^{ij}(z)=[\mu_\pm(0,z)-I]_{ij}$, $i,j=1,2,3$ for concise, then it belongs in $C_B^0(\mathbb{R})\cap L^2(\mathbb{R})$.
(\ref{scatteringcoefficient2}) leads to
\begin{align}
	s_{11}(z)=&\det\left([\mu_-]_1(0,z),[\mu_-]_2(0,z),[\mu_+]_1(0,z) \right) \nonumber\\
	=&(m_-^{11}+1)(m_-^{22}+1)m_+^{31}-(m_-^{11}+1)m_-^{32}m_+^{21}-m_-^{21}m_-^{12}m_+^{31}\nonumber\\
	&+m_-^{21}m_-^{32}(m_+^{11}+1)+m_-^{31}m_-^{12}m_+^{21}-m_-^{31}(m_-^{22}+1)(m_+^{11}+1),\\
	s_{12}(z)=&\det\left([\mu_+]_2(0,z),[\mu_-]_2(0,z),[\mu_-]_3(0,z) \right) \nonumber\\
	=&(m_-^{33}+1)(m_-^{22}+1)m_+^{12}-(m_+^{22}+1)m_-^{13}m_-^{32}-m_-^{23}m_+^{12}m_-^{32}\nonumber\\
	&+m_-^{13}m_+^{32}(m_-^{22}+1)+m_-^{23}m_-^{12}m_+^{32}-m_-^{12}(m_+^{22}+1)(m_-^{33}+1).
\end{align}
Then proposition \ref{mu1}  gives the boundedness of $s_{11}(z)$, $s_{11}'(z)$, $s_{12}(z)$, $s_{12}'(z)$ and the $L^2$-integrability of $s_{12}(z)$, $s_{12}'(z)$. So we just need to show $zs_{12}(z)\in L^2(\mathbb{R})$. From (\ref{lax3.1}), we obtain
\begin{align}
	[\mu_+]_2&(0,z)-e_2=(m_+^{12}(z),m_+^{22}(z),m_+^{32}(z))^T= \int_{0}^{+\infty}\left(\begin{array}{ccc}
		e^{-iz(a_1-a_2)y}p_{12} \\
		0 \\ -e^{iz(a_1-a_3)y}\bar{p}_{23}
	\end{array}\right)dy\nonumber\\[8pt]
	&+\int_{0}^{+\infty}\left(\begin{array}{ccc}
		0 & e^{-iz(a_1-a_2)y}p_{12}  & e^{-iz(a_1-a_2)y}p_{13}\\
		-\bar{p}_{12} & 0 & p_{23}\\
		-e^{iz(a_1-a_3)y}\bar{p}_{13} & -e^{iz(a_1-a_3)y}\bar{p}_{23} & 0
	\end{array}\right)\left(\begin{array}{ccc}
	m_+^{12} \\
	m_+^{22} \\
	m_+^{32}
\end{array}\right)dy.\nonumber
\end{align}
So
\begin{align}
	zm_+^{12}=&z\int_{0}^{+\infty}e^{-iz(a_1-a_2)y}p_{12}dy\nonumber\\
	&+z\int_{0}^{+\infty}e^{-iz(a_1-a_2)y} p_{12}m_+^{22}+z\int_{0}^{+\infty}e^{-iz(a_1-a_2)y}p_{13}m_+^{32} dy\nonumber\\
	&\triangleq H_1+H_2+H_3.
\end{align}
By integration by parts   and (\ref{lax1.1}),  we can further calculate
\begin{align}
	-i(a_1-a_2)H_1&=p_{12}(0)-\int_{0}^{+\infty}e^{-iz(a_1-a_2)y}p'_{12}dy,\nonumber\\
	-i(a_1-a_2)H_2&=p_{12}(0)m_+^{22}(0,z)-\int_{0}^{+\infty}e^{-iz(a_1-a_2)y}\left( p_{12}'m_+^{22}-|p_{12}|^2m_+^{12}+p_{12}p_{23}m_+^{32}\right) dy
	\nonumber\\
	-i(a_1-a_2)H_3&=p_{13}(0)m_+^{32}(0,z)+i(a_2-a_3)H_3\nonumber\\
	&-\int_{0}^{+\infty}e^{-iz(a_1-a_2)y}\left( p_{13}'m_+^{32}-|p_{13}|^2m_+^{12}-p_{13}\bar{p}_{23}(m_+^{22}+1)\right) dy.
\end{align}
By Lemma \ref{lemma2} and proposition \ref{mu1}, $zm_+^{12}$ can be expressed as
\begin{align}
	zm_+^{12}(z)=\frac{ip_{12}(0)}{(a_1-a_2)}+H_+^{12}(z),
\end{align}
where $H_+^{12}(z)$ is a $L^2$-integrable and bounded function on $\mathbb{R}$. In fact,  we claim that  $zm_\pm^{ij}(z)=\frac{ip_{ij}(0)}{(a_i-a_j)}+H_\pm^{ij}(z)$ with a $L^2$-integrable and bounded function $H_\pm^{ij}(z)$, which give the  $L^2$-integrability of $zs_{12}(z)$ on $\mathbb{R}$.
Using above results, we then   finally obtain the following proposition.

\begin{Proposition}\label{pror}
If the initial data $p_{ij }(x) \in  H^{1,2}(\mathbb{R})$, then $r_j(z)\in H^{1,1}(\mathbb{R}), \ j=1,2,3,4$.
\end{Proposition}

\section{The deformation of   RH problem}\label{sec3}

\quad
The long-time asymptotic  of RHP \ref{RHP1}  is affected by the growth and decay of the exponential function $e^{\pm2it z\theta_{ij} }$ appearing in both the jump relation and the residue conditions. So we need control the real part of $\pm itz\theta_{ij}$.
In this section, we introduce  a new transform  $M(z)\to M^{(1)}(z)$,  which  make that the  $M^{(1)}(z)$ is well behaved as $t\to \infty$ along any characteristic line.
The growth and  decay properties  of $e^{itz\theta_{ij}}$ as $t\to \infty$ is determined by
\begin{align}
&\text{Re}(itz\theta_{ij})=-t\text{Im}(z)\theta_{ij} .\label{Reitheta}
\end{align}
So when $z\in \mathbb{C}^+$, the asymptotic behavior  of $e^{itz\theta_{ij}}$ only depends on the sign of $\theta_{ij}=(a_i-a_j)\xi+(b_i-b_j)$, namely, the   sign of $\xi+n_{ij}$. And the jump matrix $V(z)$  in (\ref{jumpv})  needs to  be restricted according to the sign of $\text{Re}(itz\theta_{ij})$.   Unlike the case of  $2\times2$  jump matrix such  as  NLS equation \cite{RN10},
it is complicated  to divide a  $3\times3$  jump matrix into  a product of  upper and downer triangular matrix. At present  paper, we only consider
$$ {\rm  Case\ I.} \ \    \xi>-n_{12}>-n_{13}>-n_{23}; \ \  {\rm Case\ II.} \  \ -n_{23}<-n_{13}<\xi<-n_{12}. $$
 In these two cases, the  jump matrix $V(z)$    can be  split   into  the product of  two simple upper and downer triangular matrices
\begin{align}
	&e^{-iz(x\widehat{A}+t\widehat{B})}V(z)=\left(\begin{array}{ccc}
		1 & 0 & 0\\ 	
		-\bar{r}_1e^{-iz\theta_{12}} & 1 & 0 \\
		r_2e^{-iz\theta_{13}} & r_3e^{-iz\theta_{23}}   & 1
	\end{array}\right)\left(\begin{array}{ccc}
	1 & -r_1e^{iz\theta_{12}} & \bar{r}_2e^{iz\theta_{13}}\\ 	
	0 & 1 & \bar{r}_3e^{iz\theta_{23}} \\
	0 & 0   & 1
	\end{array}\right)\\
	&=\left(\begin{array}{ccc}
		1 & \dfrac{-r_1e^{iz\theta_{12}}}{1+|r_1|^2} & 0\\ 	
		0 & 1 & 0 \\
		-\bar{r}_4e^{-iz\theta_{13}} & \dfrac{r_3-r_1r_2}{1+|r_1|^2}e^{-iz\theta_{23}}   & 1
	\end{array}\right)(1+|r_1|^2)^{\sigma_3}\left(\begin{array}{ccc}
	1 & 0 & -r_2e^{iz\theta_{13}}\\ 	
	-\frac{\bar{r}_1e^{-iz\theta_{12}}}{1+|r_1|^2} & 1 & \frac{\bar{r}_3-\bar{r}_1\bar{r}_2}{1+|r_1|^2}e^{iz\theta_{23}} \\
	0 & 0   & 1
\end{array}\right),
\end{align}
where $\sigma_3=$diag$(-1,1,0)$. In fact, the other cases($\xi<-n_{13}$) also have analogous
factorizations, but  they  are  complicated  and   calumniation will be  tedious. So we do not dicuss it in this paper.
We will utilize  these factorizations to deform the jump contours, so that the oscillating factor $e^{\pm it\theta_{ij}}$ are decaying in corresponding region, respectively.
For brevity, we denote
$$\mathcal{N}\triangleq\left\lbrace 1,...,N_1+N_2\right\rbrace, \ \ \mathcal{N}_1\triangleq\left\lbrace 1,...,N_1\right\rbrace, \ \
\mathcal{N}_2\triangleq\left\lbrace N_1+1,...,N_1+N_2\right\rbrace, $$
and define   partitions $\Delta(\xi)$ and $\nabla(\xi)$
\begin{align}
&\Delta(\xi)=\left\{ \begin{array}{ll}	
	\emptyset,   &\text{as }\xi>-n_{12};\\[12pt]
	\left\lbrace 1,...,N_1\right\rbrace  , &\text{as }-n_{13}<\xi<-n_{12};\\
\end{array}\right.,
\nabla(\xi)=\mathcal{N}\setminus\nabla.\label{devide}
\end{align}
For $z_n$ with $n\in\Delta(\xi)$, the residue of $M(z)$ at $z_n$ in (\ref{RES1}) grows without bound as $t\to\infty$. Similarly, for $z_n$ with $n\in\nabla$, the residue are bounded or approaching to be $0$. Denote a small positive constant
\begin{equation}
	\rho_0=\min_{ 1\leq i<j\leq3}|\theta_{ij}|>0.\label{rho0}
\end{equation}
Note that, $\theta_{12}$ has different identities for $\xi>-n_{12}$ and $-n_{13}<\xi<-n_{12}$. Namely,  the functions which  will be used  following depend on $\xi$. Denote
\begin{align}
	I(\xi)=\left\{ \begin{array}{ll}
		\emptyset,   &\text{as } \xi>-n_{12},\\[12pt]
		\mathbb{R} , &\text{as }-n_{13}<\xi<-n_{12}.\\
	\end{array}\right.
\end{align}
Define  functions
\begin{align}
&\delta (z)=\delta (z,\xi)=\exp\left(i\int _{I(\xi)}\dfrac{\nu(s)}{s-z}ds\right), \ \nu(s)=-\frac{1}{2\pi}\log (1+|r_1(s)|^2),\\
&T(z)=T(z,\xi)=\prod_{n\in \Delta(\xi)}\dfrac{z-z_n}{z-\bar{z}_n}\delta (z,\xi)\label{T}.
\end{align}
In  the above formulas, we choose the principal branch of power and logarithm functions. Obviously, for the  Case I.  $\xi>-n_{12}$,  we have $T(z,\xi)\equiv 1$.
\begin{Proposition}\label{proT}
	The function  $T(z,\xi)$  defined by (\ref{T}) in  Case  $-n_{13}<\xi<-n_{12}$ has following properties:\\
{\rm	(a)} \ $T$ is meromorphic in $\mathbb{C}\setminus \mathbb{R}$, and for each $n\in\Delta(\xi)$, $T(z)$ has a simple pole at $\bar{z}_n$ and a simple zero at $z_n$;\\
{\rm	(b)} \  For $z\in I(\xi)$, as $z$ approaching the real axis from above and below, $T$ has boundary values $T_\pm$, which satisfy:
	\begin{equation}
		T_+(z)=(1+|r_1(z)|^2)T_-(z),\hspace{0.5cm}z\in I(\xi);
	\end{equation}
{\rm	(c)} \  $\lim_{z\to \infty}T(z)=1$, and when $|\arg(z)|\leq c<\pi$,
	\begin{align}
		T(z,\xi)=1+iT_1(\xi)\frac{1}{z} +\mathcal{O}(z^{-2}) ;
	\end{align}
where
$$T_1(\xi)= 2\sum_{n\in\Delta(\xi)}  {\rm Im}(z_n)-\int_{I(\xi)}\nu(s)ds.$$
{\rm	(d)} \  As $z\to 0$, along $z=\rho e^{i\psi}$, $\rho>0$, $|\psi|\leq c<\pi$
	\begin{equation}
	|T(z,\xi)-T(0,\xi)|\leq |z|^{1/2}
	\label{T0}.
	\end{equation}
\end{Proposition}
\begin{proof}
	 Properties (a)  can be obtain by simple calculation from the definition of $T(z)$ in (\ref{T}). And (b)  follows from the Plemelj formula. By the Laurent expansion (c) can be obtained immediately. For brevity, we  omit computation. As for (d),
	 \begin{align}
	& 	|T(z,\xi)-T(0,\xi)|\leq \Bigg|\prod_{n\in \Delta(\xi)}\dfrac{z-z_n}{z-\bar{z}_n}\prod_{n\in \Delta(\xi)}\dfrac{\bar{z}_n}{z_n}\exp\left(i\int _{I(\xi)}\nu(s)(\dfrac{1}{s-z}-\frac{1}{s})ds \right)-1 \Bigg| \bigg|\prod_{n\in \Delta(\xi)}\dfrac{z_n}{\bar{z}_n}\bigg|\nonumber\\
	 	&\lesssim |\int _{I(\xi)}\nu(s)(\dfrac{1}{s-z}-\frac{1}{s})ds|\lesssim \parallel \nu'\parallel_2|z|^{1/2}\lesssim \parallel r'_1\parallel_2|z|^{1/2}	\lesssim |z|^{1/2}.	\nonumber
	 \end{align}
\end{proof}

We define  a  new  matrix-valued   function $M^{(1)}(z)$ by
\begin{equation}
M^{(1)}(z)=M(z)T(z)^{\sigma_3},\label{transm1}
\end{equation}
which then satisfies the following RH problem.

\begin{RHP}\label{RHP3}
	Find a matrix-valued function  $  M^{(1)}(z )=M^{(1)}(z;x,t)$ which satisfies:
	
	$\blacktriangleright$ Analyticity: $M^{(1)}(z)$ is meromorphic in $\mathbb{C}\setminus \mathbb{R}$;

	$\blacktriangleright$ Jump condition: $M^{(1)}(z)$ has continuous boundary values $M^{(1)}_\pm(z)$ on $\mathbb{R}$ and
	\begin{equation}
		M^{(1)}_+(z)=M^{(1)}_-(z)V^{(1)}(z),\hspace{0.5cm}z \in \mathbb{R},
	\end{equation}
	where when $z\in \mathbb{R}\setminus I(\xi)$,
	\begin{equation}
		V^{(1)}(z)=
			\left(\begin{array}{ccc}
				1 & 0 & 0\\ 	
				-\bar{r}_1e^{-iz\theta_{12}} & 1 & 0 \\
				r_2e^{-iz\theta_{13}} & r_3e^{-iz\theta_{23}}   & 1
			\end{array}\right)\left(\begin{array}{ccc}
				1 & -r_1e^{iz\theta_{12}} & \bar{r}_2e^{iz\theta_{13}}\\ 	
				0 & 1 & \bar{r}_3e^{iz\theta_{23}} \\
				0 & 0   & 1
			\end{array}\right),   \label{jumpv1}
	\end{equation}
	when $z\in I(\xi)$,
	\begin{equation}
		V^{(1)}(z)=
			\left(\begin{array}{ccc}
				1 & \dfrac{-r_1e^{iz\theta_{12}}}{1+|r_1|^2} & 0\\ 	
				0 & 1 & 0 \\
				-\bar{r}_4e^{-iz\theta_{13}} & \dfrac{r_3-r_1r_2}{1+|r_1|^2}e^{-iz\theta_{23}}   & 1
			\end{array}\right)\left(\begin{array}{ccc}
				1 & 0 & -r_2e^{iz\theta_{13}}\\ 	
				-\frac{\bar{r}_1e^{-iz\theta_{12}}}{1+|r_1|^2} & 1 & \frac{\bar{r}_3-\bar{r}_1\bar{r}_2}{1+|r_1|^2} e^{iz\theta_{23}}\\
				0 & 0   & 1
			\end{array}\right);
	\end{equation}
	
	$\blacktriangleright$ Asymptotic behaviors:
	\begin{align}
		M^{(1)}(z) =& I+\mathcal{O}(z^{-1}),\hspace{0.5cm}z \rightarrow \infty;
\end{align}

	$\blacktriangleright$ Residue conditions: $M^{(1)}$ has simple poles at each point $z_n$ and $\bar{z}_n$ for $n\in\mathcal{N}$ with:
	\begin{align}
		&\res_{z=z_n}M^{(1)}(z)=\lim_{z\to z_n}M^{(1)}(z)\Gamma_n(\xi),\\
		&\res_{z=\bar{z}_n}M^{(1)}(z)=\lim_{z\to \bar{z}_n}M^{(1)}(z)\tilde{\Gamma}_n(\xi),
	\end{align}
	where  $\Gamma_n(\xi)$ and  $\tilde{\Gamma}_n(\xi)$  matrix defined by:
	
	for or $n=1,...,N_1$,
	\begin{align}
		\Gamma_n(\xi)=\left\{\begin{array}{lll}\left(\begin{array}{ccc}
				0 & c_ne^{iz_nt\theta_{12}} & 0\\ 	
				0 & 0 & 0 \\
				0 & 0   & 0
			\end{array}\right),&\text{ for }\xi>-n_{12};\\
			\left(\begin{array}{ccc}
				0 & 0 & 0\\ 	
				c_n^{-1}e^{-iz_nt\theta_{12}}(T^{-1})'(z_n)^{-2} & 0 & 0 \\
				0 & 0   & 0
			\end{array}\right),&\text{ for }-n_{13}<\xi<-n_{12};\end{array}\right.,\label{RES5}
	\end{align}
	\begin{align}
	\tilde{\Gamma}_n(\xi)=\left\{\begin{array}{lll}\left(\begin{array}{ccc}
			0 & 0 & 0\\ 	
			\tilde{c}_ne^{-i\bar{z}_nt\theta_{12}} & 0 & 0 \\
			0 & 0   & 0
		\end{array}\right),&\text{ for }\xi>-n_{12};\\
		\left(\begin{array}{ccc}
			0 & \tilde{c}_n^{-1}e^{i\bar{z}_nt\theta_{12}}T'(\bar{z}_n)^{-2} & 0\\ 	
			0 & 0 & 0 \\
			0 & 0   & 0
		\end{array}\right),&\text{ for }-n_{13}<\xi<-n_{12};\end{array}\right.,\label{RES6}
	\end{align}

	for $n=N_1+1,...,N_1+N_2$,
		\begin{align}
		&\Gamma_n(\xi)=
			\left(\begin{array}{ccc}
				0 & 0 & 0\\ 	
				0 & 0 & c_ne^{iz_nt\theta_{23}}T(z_n,\xi) \\
				0 & 0   & 0
			\end{array}\right),\label{RES7}\\
		&\tilde{\Gamma}_n(\xi)=
			\left(\begin{array}{ccc}
				0 & 0 & 0\\ 	
				0 & 0 & 0 \\
				0 & \tilde{c}_ne^{-i\bar{z}_nt\theta_{23}}T^{-1}(\bar{z}_n,\xi)   & 0
			\end{array}\right).\label{RES8}
	\end{align}
\end{RHP}

\begin{proof}
	  The   analyticity and symmetry of $M^{(1)}(z)$ is directly from its definition, the Proposition \ref{proT} and the identities of $M$.Then by simple calculation we can obtain the residues condition and jump condition from (\ref{RES1}), (\ref{RES2}) (\ref{jumpv}) and (\ref{transm1}). As for asymptotic behaviors, from Proposition \ref{proT} (c), we  obtain the asymptotic behaviors of $M^{(1)}(z)$.
\end{proof}

\section{A mixed $\bar{\partial}$-RH problem and decomposition}\label{sec4}

\quad  In this section,  we make continuous extension for  the jump matrix $V^{(1)}(z)$  to remove the jump from $\mathbb{R}$. Besides, the new problem is hoped to  takes advantage of the decay/growth of $e^{2itz \theta_{ij}}$ for $z\notin\mathbb{R}$. For this purpose, we  introduce new contours as follow:
\begin{align}
	&\Sigma_k=(-1)^{(k-1)/2}i\varrho+e^{(k-1)i\pi/2+\varphi}R_+,\hspace{0.5cm}k=1,3;\\
	&\Sigma_k=(-1)^{k/2+1}i\varrho+e^{ki\pi/2-\varphi}R_+,\hspace{0.5cm}k=2,4,
\end{align}
where and $\frac{\pi}{4}>\varphi>0$ is an fixed sufficiently small angle  achieving that  $\Omega_i$ for $i=1,3,4,6$ don't intersect  any of $\mathbb{D}(z_n,\varrho)$ or $\mathbb{D}(\bar{z}_n,\varrho)$. These contours together with $\mathbb{R}$ are the  boundary of  new six regions $\Omega_1$,...,$\Omega_6$. And $\Omega_i=\Omega_{i0}\cup\Omega_{i1}$ for $i=1,3,4,6$ with
\begin{align*}
	&\Omega_{10}=\left\lbrace z|\text{Re}z>0,0<\text{Im}z<\varrho \right\rbrace,\  \Omega_{30}=\left\lbrace z|\text{Re}z<0,0<\text{Im}z<\varrho \right\rbrace,\\
	&\Omega_{11}=\left\lbrace z|0<\arg(z-i\varrho)<\varphi \right\rbrace,\  \Omega_{31}=\left\lbrace z|\pi-\varphi<\arg(z-i\varrho)<\pi \right\rbrace,\\
	&\Omega_{2}=\left\lbrace z|\varphi<\arg(z-i\varrho)<\pi-\varphi \right\rbrace,
\end{align*}
and $\Omega_{5}$,  $\Omega_{60}$, $\Omega_{40}$, $\Omega_{61}$, $\Omega_{41}$ is the  conjugate of $\Omega_{2}$, $\Omega_{10}$, $\Omega_{30}$, $\Omega_{11}$, $\Omega_{31}$ respectively.
In addition, let
\begin{align}
	&\Omega=\underset{k=1,3,4,6}{\cup}\Omega_k,\  \ \Sigma^{(2)}=\Sigma_1\cup\Sigma_2\cup\Sigma_{3}\cup\Sigma_{4}  ,
\end{align}
which are shown in Figure \ref{figR2}.
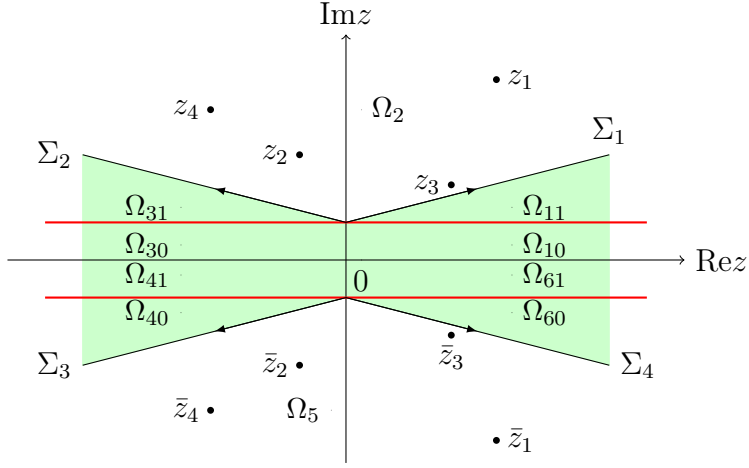
\begin{figure}[H]
	\centering
		\begin{tikzpicture}[node distance=2cm]
		\draw[green!20, fill=green!20] (0,-0.5)--(3.5,-1.4)--(3.5,1.4)--(0,0.5)--(0,-0.5)--(-3.5,-1.4)--(-3.5,1.4)--(0,0.5);
		\draw(0,0.5)--(3.5,1.4)node[above]{$\Sigma_1$};
		\draw(0,0.5)--(-3.5,1.4)node[left]{$\Sigma_2$};
		\draw(0,-0.5)--(-3.5,-1.4)node[left]{$\Sigma_3$};
		\draw(0,-0.5)--(3.5,-1.4)node[right]{$\Sigma_4$};
		\draw[->](-4.5,0)--(4.5,0)node[right]{ Re$z$};
		\draw[->](0,-2.7)--(0,3)node[above]{ Im$z$};
		\draw[-][thick,red](-4,0.5)--(4,0.5);
		\draw[-][thick,red](-4,-0.5)--(4,-0.5);
		\draw[-latex](0,-0.5)--(-1.75,-0.95);
		\draw[-latex](0,0.5)--(-1.75,0.95);
		\draw[-latex](0,0.5)--(1.75,0.95);
		\draw[-latex](0,-0.5)--(1.75,-0.95);
		\coordinate (I) at (0.2,0);
		\coordinate (C) at (-0.2,2.2);
		\coordinate (D) at (2.2,0.2);
		\fill (D) circle (0pt) node[right] {\small$\Omega_{10}$};
		\coordinate (J) at (-2.2,-0.2);
		\fill (J) circle (0pt) node[left] {\small$\Omega_{41}$};
		\coordinate (k) at (-2.2,0.2);
		\fill (k) circle (0pt) node[left] {\small$\Omega_{30}$};
		\coordinate (k) at (2.2,-0.2);
		\fill (k) circle (0pt) node[right] {\small$\Omega_{61}$};
		\fill (I) circle (0pt) node[below] {$0$};
		\coordinate (a) at (2.2,0.7);
		\fill (a) circle (0pt) node[right] {\small$\Omega_{11}$};
		\coordinate (b) at (0.2,2);
		\fill (b) circle (0pt) node[right] {\small$\Omega_{2}$};
		\coordinate (c) at (-2.2,0.7);
		\fill (c) circle (0pt) node[left] {\small$\Omega_{31}$};
		\coordinate (d) at (-2.2,-0.7);
		\fill (d) circle (0pt) node[left] {\small$\Omega_{40}$};
		\coordinate (e) at (-0.2,-2);
		\fill (e) circle (0pt) node[left] {\small$\Omega_{5}$};
		\coordinate (f) at (2.2,-0.7);
		\fill (f) circle (0pt) node[right] {\small$\Omega_{60}$};
		\coordinate (A) at (2,2.4);
		\coordinate (B) at (2,-2.4);
		\coordinate (C) at (-0.616996232,1.4);
		\coordinate (D) at (-0.616996232,-1.4);
		\coordinate (E) at (1.4,1);
		\coordinate (F) at (1.4,-1);
		\coordinate (G) at (-1.8,2);
		\coordinate (H) at (-1.8,-2);
		\fill (A) circle (1.3pt) node[right] {$z_1$};
	\fill (B) circle (1.3pt)  node[right] {$\bar{z}_1$};
	\fill (C) circle (1.3pt) node[left] {$z_2$};
	\fill (D) circle (1.3pt) node[left] {$\bar{z}_2$};
	\fill (E) circle (1.3pt) node[left] {$z_3$};
	\fill (F) circle (1.3pt) node[below] {$\bar{z}_3$};
	\fill (G) circle (1.3pt) node[left] {$z_4$};
	\fill (H) circle (1.3pt) node[left] {$\bar{z}_4$};
		\end{tikzpicture}
	\caption{The green  region is $\Omega$, red  lines are two critical lines  Im$z=\pm  \varrho$, which   divide  $\Omega_i$,   $i=1,3,4,6$
 into  two part $\Omega_{i0}$ and $\Omega_{i1}$ respectively.
 }
	\label{figR2}
\end{figure}

Introduce following functions for brief:
\begin{align}
	&q_{11}(z,\xi)=\left\{\begin{array}{lll}\dfrac{\bar{r}_1 }{1+|r_1|^2},\text{ for }-n_{13}<\xi<-n_{12};\\
	r_1(z),\text{ for }\xi>-n_{12};\end{array}\right.,\\
	&q_{12}(z,\xi)=\left\{\begin{array}{lll}\dfrac{\bar{r}_3-\bar{r}_1\bar{r}_2 }{1+|r_1|^2},\text{ for }-n_{13}<\xi<-n_{12};\\
	\bar{r}_3,\text{ for }\xi>-n_{12};\end{array}\right.,\\
	&q_{13}(z,\xi)=\left\{\begin{array}{lll}-r_2,\text{ for }-n_{13}<\xi<-n_{12};\\
	\bar{r}_2,\text{ for }\xi>-n_{12};\end{array}\right.,\\
	&q_{21}(z,\xi)=\left\{\begin{array}{lll}\dfrac{-r_1 }{1+|r_1|^2},\text{ for }-n_{13}<\xi<-n_{12};\\
	-\bar{r}_1,\text{ for }\xi>-n_{12};\end{array}\right.,\\
&q_{22}(z,\xi)=\left\{\begin{array}{lll} -\bar{r}_4,\text{ for }-n_{13}<\xi<-n_{12};\\
	r_2,\text{ for }\xi>-n_{12};\end{array}\right.,\\
&q_{23}(z,\xi)=\left\{\begin{array}{lll}\dfrac{r_3-r_1r_2 }{1+|r_1|^2},\text{ for }-n_{13}<\xi<-n_{12};\\
	r_3,\text{ for }\xi>-n_{12};\end{array}\right..
\end{align}
Besides, from $r_i\in W^{2,2}(\mathbb{R})$, it also has that $q_{11}'(z)$, $q_{12}'(z)$, $q_{21}'(z)$ and $q_{23}'(z)$  exist and are in $L^2(\mathbb{R})\cup L^\infty(\mathbb{R})$. And $\parallel q_{11}'(z)\parallel_p\lesssim \parallel r'(z)\parallel_p$ for $p=2,\infty$.
Then the next step is to construct a matrix function $R^{(2)}$. We need to remove jump on $\mathbb{R}$, and  have some mild control on $\bar{\partial}R^{(2)}$ sufficient to ensure that the $\bar{\partial}$-contribution to the long-time asymptotics of $p_{ij}(x, t)$ is negligible.

Let
\begin{equation}
	\varrho=\frac{1}{3}\min\left\lbrace\min_{ j\neq i\in \mathcal{N}}|z_i-z_j|, \min_{j\in \mathcal{N}}\left\lbrace |\text{Im}z_j|  \right\rbrace  \right\rbrace,\label{varrho}
\end{equation}
and Introduce two functions $\mathcal{X}_1(x)\in C_0^\infty$ and $\mathcal{X}_2(x)\in C^\infty$ with
\begin{align}
	\mathcal{X}_1(x)=\left\{\begin{array}{lll}
		1,\ x\leq 0;\\
		0,\ x\geq \varrho;\end{array}\right.,\  \	\mathcal{X}_2(x)=\left\{\begin{array}{lll}
		1,\ \ |x|\leq 1;\\
		x^{-1},\ |x|\geq 2.\end{array}\right. \label{X2}
\end{align}
In addtion, $|\mathcal{X}_1(x)|$, $|\mathcal{X}_2(x)|\leq1$. Note that $\theta_{12}(z)$ has different property in the cases of $\xi>-n_{12}$ and $-n_{13}<\xi<-n_{12}$, so the construction of $R^{(2)}(z)$ depend on $\xi$.
Then we choose $R^{(2)}(z,\xi)$ as:

Case I: for $\xi=\frac{x}{t}>-n_{12}$,
\begin{equation}
R^{(2)}(z,\xi)=\left\{\begin{array}{lll}
\left(\begin{array}{ccc}
1 & R_{11}(z,\xi)e^{it\theta_{12}}& R_{13}(z,\xi)e^{it\theta_{13}}\\
0 & 1& R_{12}(z,\xi)e^{it\theta_{23}}\\
0 & 0 & 1
\end{array}\right), & z\in \Omega_j,j=1,3;\\
\\
\left(\begin{array}{ccc}
	1 & 0&0 \\
	R_{21}(z,\xi)e^{-it\theta_{12}} & 1&0 \\
	R_{22}(z,\xi)e^{-it\theta_{13}} & R_{23}(z,\xi)e^{-it\theta_{23}} & 1
\end{array}\right),  &z\in \Omega_j,j=6,4;\\
\\
I,  &elsewhere;\\
\end{array}\right.\label{R(2)-}
\end{equation}

Case II: for $\xi=\frac{x}{t}\in(-n_{13},-n_{12})$,
\begin{equation}
	R^{(2)}(z,\xi)=\left\{\begin{array}{lll}
		\left(\begin{array}{ccc}
			1 & 0& R_{13}(z,\xi)e^{it\theta_{13}}\\
			R_{11}(z,\xi)e^{-it\theta_{12}} & 1& R_{12}(z,\xi)e^{it\theta_{23}}\\
			0 & 0 & 1
		\end{array}\right), & z\in \Omega_j,j=1,3;\\
		\\
		\left(\begin{array}{ccc}
			1 & R_{21}(z,\xi)e^{it\theta_{12}}&0 \\
			0 & 1&0 \\
			R_{22}(z,\xi)e^{-it\theta_{13}} & R_{23}(z,\xi)e^{-it\theta_{23}} & 1
		\end{array}\right),  &z\in \Omega_j,j=6,4;\\
		\\
		I,  &elsewhere;\\
	\end{array}\right.\label{R(2)+}
\end{equation}

where  the functions $R_{ij}$, $i=1,2$, $j=1,2,3$, are defined in following proposition.
\begin{Proposition}\label{proR}
	 $R_{1j}$: $\bar{\Omega}_1\cup\bar{\Omega}_3\to \mathbb{C}$ and $R_{2j}$: $\bar{\Omega}_4\cup\bar{\Omega}_6\to \mathbb{C}$, $j=1,2,3$ have boundary values as follow:
	\begin{align}
	&R_{1j}(z,\xi)=\Bigg\{\begin{array}{ll}
	q_{1j}(z,\xi)T_+(z)^{m_{1j}} & z\in \mathbb{R},\\
	q_{1j}(0,\xi)T(0)^{m_{1j}}\mathcal{X}_2(|z-i\varrho|) &z\in \Sigma_1\cup\Sigma_2,\\
	\end{array} , \\
	&R_{2j}(z,\xi)=\Bigg\{\begin{array}{ll}
	q_{2j}(z,\xi)T_-(z)^{m_{2j}} & z\in \mathbb{R},\\
	q_{2j}(0,\xi)T(0)^{m_{2j}}\mathcal{X}_2(|z-i\varrho|) &z\in \Sigma_3\cup\Sigma_4,\\
\end{array} ,
	\end{align}	
with $m_{11}=-m_{21}=2$, $m_{12}=-m_{22}=1$, $m_{13}=-m_{23}=-1$. And $R_{ij}(z,\xi)$  have following property:

\noindent (1) 	
\begin{align}
		|R_{ij}(z,\xi)|\lesssim \parallel q_{ij}\parallel_\infty.\label{boundR}
	\end{align}
(2) 	For $z\in\Omega_{k_i0}$, $k_1=1,3$, $k_2=4,6$
	\begin{align}
		&|\bar{\partial}R_{ij}(z,\xi)|\lesssim|q_{ij}'(\text{Re}z)|+|q_{ij}(\text{Re}z)|.\label{dbarRj}
	\end{align}
(3)  For  $z\in\Omega_{k_i1}$, $k_1=1,3$, $k_2=4,6$
	\begin{align}
	&|\bar{\partial}R_{ij}(z,\xi)|\lesssim|q_{ij}'(|z-i\varrho|)|+|\mathcal{X}_2'(|z-i\varrho|)|+|z-i\varrho|^{-1/2}.\label{dbarRj2}
\end{align}
(4)     For $z\in \Omega_2\cup\Omega_5$,
	\begin{equation}
	\bar{\partial}R^{(2)}(z)=0,
	\end{equation}
\end{Proposition}

\begin{proof}
We only give the  detail  proof  for   $R_{11}(z)$ as an example. The extensions of $R_{11}(z)$ can be constructed by:
	
	1. for $z\in \Omega_{10}$,
	\begin{align}
		R_{11}(z,\xi)=q_{11}(\text{Re}z,\xi)T(z,\xi)^{2}\mathcal{X}_1(\text{Im}z)+(1-\mathcal{X}_1(\text{Im}z))q_{11}(\text{Re}z,\xi)\mathcal{X}_2(\text{Re}z)T(z,\xi)^{2};\label{R111}
	\end{align}

	2. for $z\in \Omega_{11}$,
	\begin{align}
		R_{11}(z,\xi)=&q_{11}(|z-i\varrho|,\xi)\mathcal{X}_2(|z-i\varrho|)T(z,\xi)^{2}\cos[k_0\arg(z-i\varrho)]\nonumber\\
		&+\{1-\cos[k_0\arg(z-i\varrho)]\}q_{11}(0)T(0,\xi)^{2}\mathcal{X}_2(|z-i\varrho|),
	\end{align}
	where $k_0$ is a positive constant defined by $k_0=\frac{\pi}{2\varphi}$. The other cases are easily inferred.  Obviously,  $R_{11}$ is bounded and admit (\ref{boundR}). For $z\in \Omega_{10}$,
	denote $z=s+yi$, $s,y\in\mathbb{R}$ with $\bar{\partial}=\frac{1}{2}(\partial_s+\partial_yi)$. Then the $\bar{\partial}$-derivative of (\ref{R111}) becomes:
	\begin{align}
		2\bar{\partial}R_{11}(z,\xi)=&T(z,\xi)^{2}q_{11}'(s,\xi)\mathcal{X}_1(y)+T(z,\xi)^{2}q_{11}'(s,\xi)(1-\mathcal{X}_1(y))\mathcal{X}_2(s)\nonumber\\
		&\mathcal{X}_2'(s)q_{11}(s,\xi)(1-\mathcal{X}_1(y))T(z,\xi)^{2}+\mathcal{X}_1'(y)q_{11}(s,\xi)(1-\mathcal{X}_2(s))T(z,\xi)^{2}i,
	\end{align}
	which immediately leads to (\ref{dbarRj}). And for $z\in \Omega_{11}$,
	denote $z=\zeta+i\varrho$ with $\zeta=le^{i\phi}$. Then we have $\bar{\partial}_z=\bar{\partial}_\zeta=\frac{1}{2}e^{i\phi}\left(\partial_l+\frac{i}{l} \partial_\phi\right) $. So
	\begin{align}
	\bar{\partial}R_1(z)=&\frac{e^{i\phi}}{2}\left[  q_{11}'(l)T^2(z,\xi)\mathcal{X}_2(l)+\left( q_{11}(l)T^2(z,\xi)-q_{11}(0)T^2(0,\xi)\right) \mathcal{X}_2'(l)\right]\cos(k_0\phi)\nonumber\\
	&-\frac{e^{i\phi}i}{2l}k_0\sin(k_0\phi)\left( q_{11}(l)T^2(z,\xi)-q_{11}(0)T^2(0,\xi)  \right)\mathcal{X}_2(l)  .
	\end{align}
	So
	\begin{align}
		&|\bar{\partial}R_{ij}(z,\xi)|\lesssim|q_{ij}'(|z-i\varrho|)|+|\mathcal{X}_2'(|z-i\varrho|)|+\frac{|\mathcal{X}_2(|z-i\varrho|)||q_{ij}(|z-i\varrho|)-q_{ij}(0)|}{|s-i\varrho|}.\label{dbarRj3}
	\end{align}
	By Cauchy-Schwarz inequality, we obtain
	\begin{equation}
	|q_{11}(l)|=  |q_{11}(l)-q_{11}(0)|=|\int_{0}^lq_{11}'(s)ds|\leq \parallel q_{11}'(s)\parallel_{L^2} l^{1/2}\lesssim l^{1/2}.\label{r-}
	\end{equation}	
	And note that $T(z)$ is a bounded function in $\bar{\Omega}_1$ with estimation (\ref{T0}). Then (\ref{dbarRj2})  follows immediately.
\end{proof}

We now  use $R^{(2)}$ to define the new transformation \begin{equation}
	M^{(2)}(z)=M^{(1)}(z)R^{(2)}(z)\label{transm2},
\end{equation}
which satisfies the following mixed $\bar{\partial}$-RH problem.

\begin{RHP}\label{RHP4}
Find a matrix valued function  $M^{(2)}(z)= M^{(2)}(z;x,t)$ with following properties:

$\blacktriangleright$ Analyticity:  $M^{(2)}(z)$ is continuous in $\mathbb{C}$,  sectionally continuous first partial derivatives in
$\mathbb{C}\setminus \left( \Sigma^{(2)}\cup \left\lbrace z_n,\bar{z}_n \right\rbrace_{n\in\mathcal{N}} \right) $  and meromorphic out $\bar{\Omega}$;

$\blacktriangleright$ Jump condition: $M^{(2)}(z)$ has continuous boundary values $M^{(2)}_\pm(z)$ on $\Sigma^{(2)}$ and
\begin{equation}
	M^{(2)}_+(z)=M^{(2)}_-(z)V^{(2)}(z),\hspace{0.5cm}z \in \Sigma^{(2)},
\end{equation}
where
\begin{equation}
	V^{(2)}(z)=\left\{ \begin{array}{ll}
		R^{(2)}(z,\xi)|_{\Omega_1\cup\Omega_3},   &\text{as } 	z\in\Sigma_1\cup\Sigma_2;\\[12pt]
		R^{(2)}(z,\xi)^{-1}|_{\Omega_4\cup\Omega_6},   &\text{as } z\in\Sigma_3\cup\Sigma_4;\\
	\end{array}\right.;\label{jumpv2}
\end{equation}

$\blacktriangleright$ Asymptotic behaviors:
	\begin{align}
	M^{(2)}(z) =& I+\mathcal{O}(z^{-1}),\hspace{0.5cm}z \rightarrow \infty;\label{asyM2}
\end{align}

$\blacktriangleright$ $\bar{\partial}$-Derivative: For $z\in\mathbb{C}$
we have
\begin{align}
	\bar{\partial}M^{(2)}(z)=M^{(2)}(z)\bar{\partial}R^{(2)}(z,\xi),
\end{align}
where
Case I: for $\xi=\frac{y}{t}>-n_{12}$,
\begin{equation}
	\bar{\partial}R^{(2)}(z,\xi)=\left\{\begin{array}{lll}
		\left(\begin{array}{ccc}
			0 & \bar{\partial}R_{11}(z,\xi)e^{2it\theta_{12}}& \bar{\partial}R_{13}(z,\xi)e^{2it\theta_{13}}\\
			0 & 0& \bar{\partial}R_{12}(z,\xi)e^{2it\theta_{23}}\\
			0 & 0 & 0
		\end{array}\right), & z\in \Omega_j,j=1,3;\\
		\\
		\left(\begin{array}{ccc}
			0 & 0&0 \\
			\bar{\partial}R_{21}(z,\xi)e^{-2it\theta_{12}} & 0&0 \\
			\bar{\partial}R_{22}(z,\xi)e^{-2it\theta_{13}} & \bar{\partial}R_{23}(z,\xi)e^{-2it\theta_{23}} & 0
		\end{array}\right),  &z\in \Omega_j,j=4,6;\\
		\\
		0,  &elsewhere;\\
	\end{array}\right.\label{DBARR1}
\end{equation}

Case II: for $\xi=\frac{y}{t}\in(-n_{13},-n_{12})$,
\begin{equation}
	\bar{\partial}R^{(2)}(z,\xi)=\left\{\begin{array}{lll}
		\left(\begin{array}{ccc}
			0 & 0& \bar{\partial}R_{13}(z,\xi)e^{2it\theta_{13}}\\
			\bar{\partial}R_{11}(z,\xi)e^{-2it\theta_{12}} & 0& \bar{\partial}R_{12}(z,\xi)e^{2it\theta_{23}}\\
			0 & 0 & 0
		\end{array}\right), & z\in \Omega_j,j=1,3;\\
		\\
		\left(\begin{array}{ccc}
			0 & \bar{\partial}R_{21}(z,\xi)e^{2it\theta_{12}}&0 \\
			0 & 0&0 \\
			\bar{\partial}R_{22}(z,\xi)e^{-2it\theta_{13}} & \bar{\partial}R_{23}(z,\xi)e^{-2it\theta_{23}} & 0
		\end{array}\right),  &z\in \Omega_j,j=6,4;\\
		\\
		0,  &elsewhere;\\
	\end{array}\right.\label{DBARR2}
\end{equation}

$\blacktriangleright$ Residue conditions: $M^{(2)}$ has simple poles at each point $z_n$ and $\bar{z}_n$ for $n\in\mathcal{N}$   with:
\begin{align}
	&\res_{z=z_n}M^{(2)}(z)=\lim_{z\to z_n}M^{(2)}(z)\Gamma_n(\xi),\\
	&\res_{z=\bar{z}_n}M^{(2)}(z)=\lim_{z\to \bar{z}_n}M^{(2)}(z)\tilde{\Gamma}_n(\xi),
\end{align}
where  $\Gamma_n(\xi)$ and  $\tilde{\Gamma}_n(\xi)$  are given in (\ref{RES5})-(\ref{RES8}).	
\end{RHP}

\quad To solve RHP 2,  we decompose it into a model   RH  problem  for $M^{rhp}(z)$  with $\bar\partial R^{(2)}(z)\equiv0$   and a pure $\bar{\partial}$-Problem with nonzero $\bar{\partial}$-derivatives.
First  we establish  a   RH problem  for the  $M^{rhp}(z)$   as follows.

\begin{RHP}\label{RHP5}
Find a matrix-valued function  $ M^{rhp}(z)= M^{rhp}(z;x,t)$ with following properties:

$\blacktriangleright$ Analyticity: $M^{rhp}(z)$ is  meromorphic  in $\mathbb{C}\setminus \Sigma^{(2)}$;

$\blacktriangleright$ Jump condition: $M^{rhp}(z)$ has continuous boundary values $M^{rhp}_\pm(z)$ on $\Sigma^{(2)}$ and
\begin{equation}
	M^{rhp}_+(z)=M^{rhp}_-(z)V^{(2)}(z),\hspace{0.5cm}z \in \Sigma^{(2)};\label{jump5}
\end{equation}

$\blacktriangleright$ $\bar{\partial}$-Derivative:  $\bar{\partial}R^{(2)}(z)=0$, for $ z\in \mathbb{C}$;

$\blacktriangleright$ Asymptotic behaviors:
	\begin{align}
	M^{rhp}(z) =& I+\mathcal{O}(z^{-1}),\hspace{0.5cm}z \rightarrow \infty;\label{asyMr}
\end{align}

$\blacktriangleright$ Residue conditions: $M^{rhp}(z)$ has simple poles at each point $z_n$ and $\bar{z}_n$ for $n\in\mathcal{N}$   with:
\begin{align}
	&\res_{z=z_n}M^{rhp}(z)=\lim_{z\to z_n}M^{rhp}(z)\Gamma_n(\xi),\\
	&\res_{z=\bar{z}_n}M^{rhp}(z)=\lim_{z\to \bar{z}_n}M^{rhp}(z)\tilde{\Gamma}_n(\xi),
\end{align}
where  $\Gamma_n(\xi)$ and  $\tilde{\Gamma}_n(\xi)$  are given in (\ref{RES5})-(\ref{RES8}).		
\end{RHP}

The unique existence  and asymptotic  of  $M^{rhp}(z)$  will shown in   section \ref{sec6}. We now use $M^{rhp}(z)$ to construct  a new matrix function
\begin{equation}
M^{(3)}(z)=M^{(2)}(z)M^{rhp}(z)^{-1}.\label{transm3}
\end{equation}
which   removes   analytical component  $M^{rhp}(z)$    to get  a  pure $\bar{\partial}$-problem.

\noindent\textbf{RHP 5}. Find a matrix-valued function  $  M^{(3)}(z)=M^{(3)}(z;x,t)$ with following identities:

$\blacktriangleright$ Analyticity: $M^{(3)}(z)$ is continuous   and has sectionally continuous first partial derivatives in $\mathbb{C}$.

$\blacktriangleright$ Asymptotic behavior:
\begin{align}
&M^{(3)}(z) \sim I+\mathcal{O}(z^{-1}),\hspace{0.5cm}z \rightarrow \infty;\label{asymbehv7}
\end{align}

$\blacktriangleright$ $\bar{\partial}$-equation:
$$\bar{\partial}M^{(3)}(z)=M^{(3)}(z) W^{(3)}(z),\ \ z\in \mathbb{C},$$
where
\begin{equation}
W^{(3)}=M^{rhp}(z)\bar{\partial}R^{(2)}(z)M^{rhp}(z)^{-1}.
\end{equation}
The unique existence  and asymptotic  of  $M^{(3)}(z)$  will shown in   section \ref{sec7}. And  in  the  RHP \ref{RHP5}, its jump matrix $ V^{(2)}(z)$ admits the  following estimates.
\begin{Proposition}\label{pro3v2}
	For the jump matrix $ V^{(2)}(z)$, we have the following  estimate
	\begin{align}
		&\parallel V^{(2)}(z)-I\parallel_{L^\infty(\Sigma^{(2)})}=\mathcal{\mathcal{O}}(e^{-t\varrho\rho_0 } ),\label{7.2}
	\end{align}
	where $\rho_0$, $\varrho$ is defined in (\ref{rho0}) and (\ref{varrho}).
\end{Proposition}
\begin{proof}
	We   prove (\ref{7.2})  for   $z\in\Sigma_{1}$,    other cases can be shown in a  similar  way.
	By using  definition of $V^{(2)}(z)$ and (\ref{boundR}),  we have
	\begin{align}
		\parallel V^{(2)}(z)-I\parallel_{L^\infty(\Sigma_1)}\leq \sum_{j=1}^{3}\parallel R_{1j}e^{-t\text{Im}z\rho_0 }\parallel_{L^\infty(\Sigma_1)}\lesssim e^{-t\varrho\rho_0 }. \label{poope}
	\end{align}
\end{proof}
\begin{corollary}\label{v2p}
	For $1\leq p\leq +\infty$, the jump matrix $V^{(2)}(z)$ satisfies
	\begin{equation}
		\parallel V^{(2)}(z)-I\parallel_{L^p(\Sigma^{(2)})}\leq K_pe^{- t\varrho\rho_0} ,
	\end{equation}
	for some constant $K_p\geq 0$ depending on $p$.
\end{corollary}
This proposition means that the jump matrix $V^{(2)}(z)$    uniformly goes to  $I$  on    $\Sigma^{(2)}$,
so  there is only exponentially small error (in t) by completely ignoring the jump condition of  $M^{rhp}(z)$. This proposition inspire us to construct
 the solution $M^{rhp}(z)$ of the RHP \ref{RHP5} in  following form
\begin{equation}
	M^{rhp}(z)=E(z)M^{sol}(z). \label{transMr}
\end{equation}
This decomposition  splits   $	M^{rhp}(z)$  into two parts:  $M^{sol}(z)$ solves a model  RHP given following obtained by ignoring the jump conditions of  RHP \ref{RHP5}, which will be solved   in next Section \ref{sec6};
And $E(z)$ is a error function, which is a solution of a small-norm RH problem and we discuss it in Section \ref{sec7}.

\begin{RHP}\label{RHPs}
 Find a matrix-valued function  $ M^{sol}(z)=M^{sol}(z;x,t)$ with following properties:

$\blacktriangleright$ Analyticity: $M^{sol}(z )$ is analytical  in $\mathbb{C}\setminus ( \mathcal{Z}\cup \bar{\mathcal{Z}})$;

$\blacktriangleright$ Asymptotic behaviors:
\begin{align}
	&M^{sol}(z ) \sim I+\mathcal{O}(z^{-1}),\hspace{0.5cm}z \rightarrow \infty;
\end{align}

$\blacktriangleright$ Residue conditions: $M^{sol}(z)$ has simple poles at each point in $ \mathcal{Z}\cup \bar{\mathcal{Z}}$ satisfying:
\begin{align}
	&\res_{z=z_n}M^{sol}(z)=\lim_{z\to z_n}M^{sol}(z)\Gamma_n(\xi),\\
	&\res_{z=\bar{z}_n}M^{sol}(z)=\lim_{z\to \bar{z}_n}M^{sol}(z)\tilde{\Gamma}_n(\xi),
\end{align}
where  $\Gamma_n(\xi)$ and  $\tilde{\Gamma}_n(\xi)$  are given in (\ref{RES5})-(\ref{RES8}).	
\end{RHP}

\section{  Asymptotic  analysis  on soliton solutions  } \label{sec6}
\quad  We will build the reflectionless case of  RHP \ref{RHP4} as RHP \ref{RHPs} to  show that  approximated  with  a finite sum  of soliton solutions in this section.
 Based on the   original RHP \ref{RHP1},  we  show the existence and uniqueness of solution of  above RHP \ref{RHPs}.
\begin{Proposition}\label{unim}
	For $M^{sol}(z )$ denoting the solution of the RHP \ref{RHPs} with scattering data $\mathcal{D}=\left\lbrace  \vec{r}(z)=(r_1(z),...,r_4(z)),\left\lbrace z_n,c_n\right\rbrace_{n\in\mathcal{N}}\right\rbrace$, $M^{sol}(z )$ exists unique. By an explicit transformation, $M^{sol}(z )$ is equivalent  to a reflectionless solution of the original RHP \ref{RHP1} with modified scattering data $\widetilde{\mathcal{D}}=\left\lbrace  0,\left\lbrace z_n,\widetilde{c}_n\right\rbrace_{n\in\mathcal{N}}\right\rbrace$, where
	\begin{equation}
		 \widetilde{c}_n =c_n\exp\left\lbrace-\frac{1}{i\pi}\int_{I(\xi)}\frac{\log(1+|r_1(s)|^2)}{s-z_n}ds \right\rbrace .
	\end{equation}
\end{Proposition}
\begin{proof}
	To transform $M^{sol}(z )$ to the soliton-solution  of RHP \ref{RHP1}, we reverses the triangularity effected in (\ref{transm1}) and (\ref{transm2})
	\begin{equation}
		N(z )= M^{sol}(z)\left( \prod_{n\in \Delta(\xi)}\dfrac{z-z_n}{z-\bar{z}_n}\right) ^{-\sigma_3}.\label{N}
	\end{equation}
then  we can verify $N(z)$ with  scattering data $ \widetilde{\mathcal{D}}$ satisfying RHP \ref{RHP1}. The transformation (\ref{N})  preserves the normalization conditions at the origin and infinity obviously. Comparing with (\ref{transm1}), this transformation  restores the influence by $T(z,\xi)$ on  residue condition in (\ref{RES5})-(\ref{RES8}), and convert them back into (\ref{RES1})-(\ref{RES4}).
	Its analyticity follows from the Proposition of $M^{rhp}(z)$ and $T(z)$ immediately.  Then $N(z )$ is solution of RHP \ref{RHP1} with absence of reflection, whose unique exact solution  exists and can be obtained as described similarly in \cite{SandRNLS} Appendix A. And the uniqueness and existences of $M^{sol}(z )$  come from (\ref{N}).
\end{proof}
From (\ref{recons u}), denote $p_{ij}^{sol}(x,t;\widetilde{\mathcal{D}})$ is the soliton solution of a reflectionless scattering data $\widetilde{\mathcal{D}}$ to the original RHP \ref{RHP1} with $$	 p_{ij}^{sol}(x,t;\widetilde{\mathcal{D}})=-i(a_i-a_j)\lim_{z\to \infty}[zN(z;\widetilde{\mathcal{D}})]_{ij}.$$ Then by (\ref{N}), it also has
\begin{align}
	p_{ij}^{sol}(x,t;\widetilde{\mathcal{D}})=-i(a_i-a_j)\lim_{z\to \infty}[zM^{sol}(z;\mathcal{D})]_{ij}.\label{pij1}
\end{align}
Although $M^{sol}(z)$ has uniqueness and  existence, not all discrete spectra have contribution as $t\to\infty$.  Give pairs points $x_1\leq x_2\in \mathbb{R}$ and velocities $v_1\leq v_2 \in \mathbb{R}$,  we  define a  cone
\begin{equation}
	C(x_1,x_2,v_1,v_2)=\left\lbrace (x,t)\in R^2:x=x_0+vt, \ x_0\in[x_1,x_2]\text{, }v\in[v_1,v_2]\right\rbrace.\label{coneC}
\end{equation}
and denote
\begin{align}
	&I=\left\lbrace v:\ v_1 <\text{Re}v <v_2 \right\rbrace, \ \ \Delta(I)=\left\lbrace n\in\mathcal{N}:\  v_{z_k}\in I\right\rbrace \nonumber \\
	&\mathcal{Z}(I)=\left\lbrace z_k\in \mathcal{Z}: \ v_{z_k}\in I\right\rbrace ,\hspace{2.1cm}N(I)=|\mathcal{Z}(I)|,
\end{align}
where $v_{z_k}$ is   velocity of  the soliton solution corresponding to   pole $z_k$ with for $k=1,...,N_1$, $v_{z_k}=-n_{12}$  and for $k=N_1,...,N_1+N_2$, $v_{z_k}=-n_{23}$.
We can show the  following proposition.

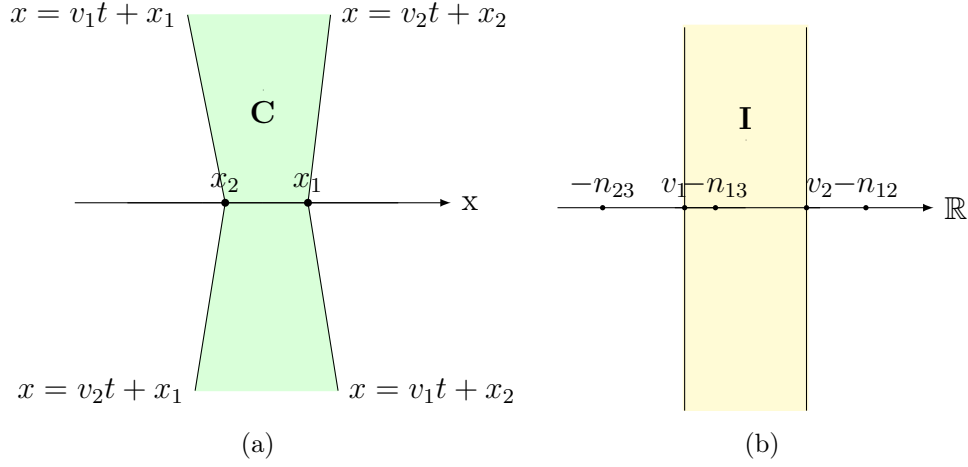
\begin{figure}[t]
	\centering
	\centering
		\subfigure[]{
		\begin{tikzpicture}[node distance=2cm]
			\draw[green!15, fill=green!15] (0.6,0)--(0.9,2.5)--(-1.0,2.5)--(-0.5,0)--(-0.9,-2.5)--(1.0,-2.5)--(0.6,0);
			\draw[-latex](-2.5,0)--(2.5,0)node[right]{x};
			\draw (0.6, 0)--(0.9, 2.5)node[right]{$x=v_2t+x_2$};
			\draw (-0.5,0)--(-0.9,-2.5)node[left]{$x=v_2t+x_1$};
			\draw (-0.5,0)--(-1.0,2.5)node[left]{$x=v_1t+x_1$};
			\draw (0.6,0)--(1.0,-2.5)node[right]{$x=v_1t+x_2$};
			\draw[-](0,0)--(-0.8,0);
			\draw[-](-0.8,0)--(-1.8,0);
			\draw[-](0,0)--(0.8,0);
			\draw[-](0.8,0)--(1.8,0);
			\coordinate (A) at (0.6,0);
			\coordinate (B) at (-0.5,0);
			\coordinate (I) at (0,1.5);
			\fill (A) circle (1.5pt) node[above] {$x_1$};
			\fill (B) circle (1.5pt) node[above] {$x_2$};
			\fill (I) circle (0pt) node[below] {$\textbf{C}$};
			\label{figzero}
		\end{tikzpicture}
	}
\subfigure[]{
		\begin{tikzpicture}[node distance=2cm]
			\filldraw[yellow!20,line width=2] (-0.8,-1.5)--(-0.8,3.6)--(0.8,3.6)--(0.8,-1.5);
			\draw[-latex](-2.5,1.2)--(2.5,1.2)node[right]{$\mathbb{R}$};
			\draw[-](-0.81,-1.5)--(-0.81,3.6);
			\draw[-](0.81,3.6)--(0.81,-1.5);
			\draw[-](0,1.2)--(-0.94,1.2)node[above] {$v_1$};
			\draw[-](0.8,1.2)--(0.99,1.2)node[above] {$v_2$};
			\coordinate (A) at (0.81,1.2);
			\coordinate (B) at (-0.81,1.2);
			\coordinate (C) at (1.6,1.2);
			\coordinate (D) at (-1.9,1.2);
			\coordinate (E) at (-0.4,1.2);
			\coordinate (I) at (0,2.1);
			\fill (A) circle (1pt);
			\fill (B) circle (1pt);
			\fill (C) circle (1pt) node[above] {$-n_{12}$};
			\fill (D) circle (1pt) node[above] {$-n_{23}$};
			\fill (E) circle (1pt) node[above] {$-n_{13}$};
			\fill (I) circle (0pt) node[above] {$\textbf{I}$};
			\label{figC}
		\end{tikzpicture}
	}
	\caption{(a) In the example here, $-n_{23}<v_1<-n_{13}<v_2<-n_{12}$, so $\mathcal{Z}(I)=\emptyset$; (b) The cone $C(x_1,x_2,v_1,v_2)$}
	\label{figC(I)}
\end{figure}
 Let $\Delta^\pm(I)$ as a partition of $\mathcal{N}$, $\Delta^\pm(I)=\Delta_1^\pm(I)\cup\Delta_2^\pm(I)$  with $\Delta_i^+(I)=\left\lbrace n\in\mathcal{N}_i:\ v_{z_n}<v_1 \right\rbrace $, and $\Delta_i^-(I)=\left\lbrace n\in\mathcal{N}_i:\ v_{z_n}>v_2 \right\rbrace $, $i=1,2$ , then $\mathcal{N}=\Delta(I)\cup\Delta^+(I)\cup\Delta^-(I)$. Denote
\begin{align}
	&N^{\Delta(I)}(z;\tilde{\mathcal{D}})=N(z;\tilde{\mathcal{D}})\Pi_1(z)\Pi_2(z),\\
	&\Pi_1(z)=\prod_{n\in \Delta_1^-(I)}\left( \dfrac{z-z_n}{z-\bar{z}_n}\right) ^{-\sigma_3},\ \Pi_2(z)=\prod_{n\in \Delta_2^-(I)}\left( \dfrac{z-z_n}{z-\bar{z}_n}\right) ^{-\sigma_2},
\end{align}
with $\sigma_2=$diag$\left\lbrace0,1,-1 \right\rbrace $. From the residue condition of $N(z;\tilde{\mathcal{D}})$ in (\ref{RES1})-(\ref{RES4}) with replacing $c_n$ by $\mathring{c}_n$,  $N^{\Delta(I)}(z;\tilde{\mathcal{D}})$ has residue as
$\res_{z=z_n}N^{\Delta(I)}(z;\tilde{\mathcal{D}})=\lim_{z\to z_n}N^{\Delta(I)}(z;\tilde{\mathcal{D}})\Gamma_n^{\Delta(I)}$ and $\res_{z=\bar{z}_n}N^{\Delta(I)}(z;\tilde{\mathcal{D}})=\lim_{z\to \bar{z}_n}N^{\Delta(I)}(z;\tilde{\mathcal{D}})\tilde{\Gamma}_n^{\Delta(I)}$. Here,
\begin{align}
	\Gamma_n^{\Delta(I)}=\left\{\begin{array}{lll}\left(\begin{array}{ccc}
			0 & c_ne^{iz_nt\theta_{12}}\Pi_2(z_n) & 0\\ 	
			0 & 0 & 0 \\
			0 & 0   & 0
		\end{array}\right),&\text{ for }n\in\mathcal{N}_1\setminus\Delta_1^-(I);\\
		\left(\begin{array}{ccc}
			0 & 0 & 0\\ 	
			c_n^{-1}e^{-iz_nt\theta_{12}}(1/\Pi_1)'(z_n)^{-2}\Pi_2(z_n)^{-1} & 0 & 0 \\
			0 & 0   & 0
		\end{array}\right),&\text{ for }n\in\Delta_1^-(I);\\
		\left(\begin{array}{ccc}
			0 & 0 & 0\\ 	
			0 & 0 & c_ne^{iz_nt\theta_{23}}\Pi_1(z_n) \\
			0 & 0   & 0
		\end{array}\right),&\text{ for }n\in\mathcal{N}_2\setminus\Delta_2^-(I);\\
		\left(\begin{array}{ccc}
			0 & 0 & 0\\ 	
			0 & 0 &  0\\
			0 & c_n^{-1}e^{-iz_nt\theta_{23}}(1/\Pi_2)'(z_n)^{-2}\Pi_1(z_n)^{-1}   & 0
		\end{array}\right),&\text{ for }n\in\Delta_2^-(I); \end{array}\right.,
\end{align}
and
\begin{align}
	\tilde{\Gamma}_n^{\Delta(I)}=\left\{\begin{array}{lll}\left(\begin{array}{ccc}
			0 & 0 & 0\\ 	
			\tilde{c}_ne^{-i\bar{z}_nt\theta_{12}}\Pi_2(\bar{z}_n)^{-1} & 0 & 0 \\
			0 & 0   & 0
		\end{array}\right),&\text{ for }n\in\mathcal{N}_1\setminus\Delta_1^-(I);\\
		\left(\begin{array}{ccc}
			0 & \tilde{c}_n^{-1}e^{i\bar{z}_nt\theta_{12}}\Pi_1'(\bar{z}_n)^{-2}\Pi_2(\bar{z}_n) & 0\\ 	
			0 & 0 & 0 \\
			0 & 0   & 0
		\end{array}\right),&\text{ for }n\in\Delta_1^-(I);\\
		\left(\begin{array}{ccc}
			0 & 0 & 0\\ 	
			0 & 0 & 0 \\
			0 & \tilde{c}_ne^{-i\bar{z}_nt\theta_{23}}\Pi_1^{-1}(\bar{z}_n)   & 0
		\end{array}\right),&\text{ for }n\in\mathcal{N}_2\setminus\Delta_2^-(I);\\
		\left(\begin{array}{ccc}
			0 & 0 & 0\\ 	
			0 & 0 & \tilde{c}_n^{-1}e^{i\bar{z}_nt\theta_{23}}\Pi_2'(\bar{z}_n)^{-2}\Pi_1(\bar{z}_n) \\
			0 &  0  & 0
		\end{array}\right),&\text{ for }n\in\Delta_2^-(I)
	\end{array}\right.,
\end{align}

For $x=x_0+v_0t$ with $x_1\leq x_0\leq x_2$ and $v_1\leq v_0\leq v_2$, consider the  exponent in  nilpotent matrix   corresponding to pole $z_n$ in residue condition (\ref{RES1}) and (\ref{RES3}), when $n=1,...,N_1$,
\begin{align}
	|e^{iz_n[(a_1-a_2)x+(b_1-b_2)t]}|=|e^{-\text{Im}z_n(a_1-a_2)x_0}||e^{-\text{Im}z_nt(a_1-a_2)(v_0-v_{z_n})}|,
\end{align}
and when $n=N_1,...,N_1+N_2$,
\begin{align}
	|e^{iz_n[(a_2-a_3)x+(b_2-b_3)t]}|=|e^{-\text{Im}z_n(a_2-a_3)x_0}||e^{-\text{Im}z_nt(a_2-a_3)(v_0-v_{z_n})}|.
\end{align}
Above equations imply that for $z_n\notin\mathcal{Z}(I)$, their residue exponentially small as
\begin{align}
	\parallel \Gamma_n^{\Delta(I)} \parallel = \mathcal{O}(e^{-a\mu(I)t}).\label{asytau}
\end{align}
So for the poles $z_n$ not in $\mathcal{Z}(I)$, we want to trap them for jumps along small closed loops enclosing themselves as $\mathbb{D}(z_n,\varrho)$ respectively and prove that this jump is  uniformly exponentially near identity. By the definition of $\varrho$ in (\ref{varrho}), for every $n\in  \mathcal{N}  $,  $\mathbb{D}(z_n,\varrho)$  are pairwise disjoint and are disjoint  with  $\mathbb{R}$. Further, from the symmetry of poles , this definition guarantee   $\mathbb{D}(\bar{z}_n,\varrho)$ have same property. To achieve this purpose,  we denote a piecewise matrix function
\begin{equation}
	G(z)=\left\{ \begin{array}{ll}
		I-\frac{\Gamma_n^{\Delta(I)}(\xi)}{z-z_n},   &\text{as } z\in\mathbb{D}(z_n,\varrho),z_n\in\mathcal{Z}\setminus\mathcal{Z}(I);\\[12pt]
		I-\frac{\tilde{\Gamma}_n^{\Delta(I)}(\xi)}{z-\bar{z}_n},   &\text{as } 	z\in\mathbb{D}(\bar{z}_n,\varrho),z_n\in\mathcal{Z}\setminus\mathcal{Z}(I);\\
		I &\text{as } 	z \text{ in elsewhere};
	\end{array}\right..\label{funcG}
\end{equation}
Introduce a new transformation as
\begin{align}
	\widetilde{m}(z;\tilde{\mathcal{D}})=N^{\Delta(I)}(z ;\tilde{\mathcal{D}} )G(z).
\end{align}
Comparing with $N^{\Delta(I)}(z ;\tilde{\mathcal{D}} )$, the new matrix function $\widetilde{m}(z;\tilde{\mathcal{D}})$ has new jump in each $\partial \mathbb{D}(\bar{z}_n,\varrho)$ which denote by $\widetilde{V}(z)$. Then by (\ref{asytau}),  direct calculation shows that  as $t\to\infty$ with $(x,t)\in C(x_1,x_2,v_1,v_2)$,
\begin{equation}
	\parallel\widetilde{V}(z)-I\parallel_{L^\infty(\widetilde{\Sigma})}=\mathcal{O}(e^{-a\mu(I)t}),\hspace{0.5cm}\widetilde{\Sigma}=\cup_{z_k\in \mathcal{Z}\setminus \mathcal{Z}(I)}\left( \partial D_k\cup\partial \bar{D}_k\right).\label{asytV}
\end{equation}
This property inspire us to consider following solution of RHP \ref{RHPs}.
\begin{Proposition}
	Denote $M^{sol}_{\mathcal{Z}(I)}(z;\mathcal{D}(I))$ as the solution of the Riemann-Hilbert problem \ref{RHPs} with scattering data $\mathcal{D}(I)=\left\lbrace  \vec{r}(z)=(r_1(z),...,r_4(z)),\left\lbrace z_n,c_n(I)\right\rbrace_{z_n\in\mathcal{Z}(I)}\right\rbrace$, where
	\begin{align}
		c_k(I)=\left\{ \begin{array}{ll}
				c_k\Pi_2(z_n), &\text{ for }n\in\mathcal{N}_1\cap \mathcal{Z}(I);\\
			c_k\Pi_1(z_n), &\text{ for }n\in\mathcal{N}_2\cap \mathcal{Z}(I);\\
		\end{array}\right.
	\end{align}
Then  there is a reflectionless solution of the original RHP \ref{RHP1} with modified scattering data $\tilde{\mathcal{D}}(I)$ defined by
	\begin{align}
		N(z;\tilde{\mathcal{D}}(I))= M^{sol}_{\mathcal{Z}(I)}(z;\mathcal{D}(I))\left( \prod_{n\in \Delta(\xi),z_n\in\mathcal{Z}(I)}\dfrac{z-z_n}{z-\bar{z}_n}\right) ^{-\sigma_3},\label{NI}
	\end{align}
	with
	\begin{align}
		\tilde{\mathcal{D}}(I)=\left\lbrace 0,\left\lbrace z_n, \widetilde{c}_n(I)\right\rbrace _{z_n\in\mathcal{Z}(I)} \right\rbrace ,\ \mathring{c}_n(I)=c_n(I)\exp\left(2i\int _{I(\xi)}\dfrac{\nu(s)ds}{s-z}\right).\label{dataI}	 
	\end{align}
\end{Proposition}
Analogously, From (\ref{recons u}), denote $p_{ij}^{sol}(x,t;\tilde{\mathcal{D}}(I))$ is the soliton solution of a reflectionless scattering data $\tilde{\mathcal{D}}(I)$ to the original RHP \ref{RHP1} with $$	 p_{ij}^{sol}(x,t;\tilde{\mathcal{D}}(I))=-i(a_i-a_j)\lim_{z\to \infty}[zN(z;\tilde{\mathcal{D}}(I))]_{ij}.$$ Then by (\ref{N}), it also has
\begin{align}
	p_{ij}^{sol}(x,t;\tilde{\mathcal{D}}(I))=-i(a_i-a_j)\lim_{z\to \infty}[zM_{\mathcal{Z}(I)}^{sol}(z;\mathcal{D}(I))]_{ij}.\label{pij2}
\end{align}
\begin{Proposition}
	For $M^{sol}(z ;\mathcal{D} )$ and $M^{sol}_{\mathcal{Z}(I)}(z;\tilde{\mathcal{D}}(I))$ defined above, as $t\to\infty$ with $(x,t)\in C(x_1,x_2,v_1,v_2)$, there exists a positive constant $a=\min\left\lbrace a_1-a_2, a_2-a_3 \right\rbrace$
	\begin{align}
		M^{sol}(z ;\mathcal{D} )=\left( I+\mathcal{O}(e^{-a\mu(I)t})\right)M^{sol}_{\mathcal{Z}(I)}(z;\tilde{\mathcal{D}}(I))\left( \prod_{n\in \Delta(\xi),z_n\notin\mathcal{Z}(I)}\dfrac{z-z_n}{z-\bar{z}_n}\right) ^{-\sigma_3},\label{DI}
	\end{align}
where
$$\mu(I)=\underset{z_k\in \mathcal{Z}\setminus \mathcal{Z}(I) }{\min}\left\lbrace \text{\rm Im}(z_k)\text{dist}( v_{z_k},\mathcal{I}) \right\rbrace=\underset{z_k\in \mathcal{Z}\setminus \mathcal{Z}(I),i=1,2}{\min}\left\lbrace \text{\rm Im}(z_k)|v_i-v_{z_k}|\right\rbrace>0.$$
\end{Proposition}
\begin{proof}
		
\quad To arrive at (\ref{DI}), we consider the asymptotic error between $N(z;\tilde{\mathcal{D}})$ and $N(z;\tilde{\mathcal{D}}(I))$.
Since  $ \widetilde{m}(z;\tilde{\mathcal{D}}) $ has same poles and residue conditions  with  $  N(z ;\tilde{\mathcal{D}}(I))$,
then
$$m_0(z)= \widetilde{m}(z;\tilde{\mathcal{D}})  N(z ;\tilde{\mathcal{D}}(I))^{-1}$$
has no poles, but  it has jump matrix for $z\in\widetilde{\Sigma}$,
\begin{equation}
	m_0^+(z)=m_0^-(z)V_{m_0}(z),
\end{equation}
where the jump matrix $V_{m_0}(z)$ given by
\begin{equation}
	V_{m_0}(z)=N(z ;\tilde{\mathcal{D}}(I))\widetilde{V}(z)N(z ;\tilde{\mathcal{D}}(I))^{-1},
\end{equation}
which, by using (\ref{asytV}),  also admits the same decaying estimate
$$\parallel V_{m_0}(z)-I\parallel_{L^\infty(\widetilde{\Sigma})}=\parallel\widetilde{V}(z)-I\parallel_{L^\infty(\widetilde{\Sigma})}=\mathcal{O}(e^{-a\mu(I)t}), \ \ t\rightarrow +\infty.$$

Then by using  the theory of small norm  RH problem,   we find  that
$m_0(z)$ exists and
$$m_0(z)=I+\mathcal{O}(e^{-a\mu(I)t}), \ \ t\to \infty,$$
which together with (\ref{funcG})  gives  the formula (\ref{DI}).
\end{proof}

\begin{corollary}\label{sol}
For the soliton solution of the reflectionless scattering data $\tilde{\mathcal{D}}$ and $\tilde{\mathcal{D}}(I)$ to the original RHP \ref{RHP1} $p_{ij}^{sol}(x,t;\tilde{\mathcal{D}})$ and $p_{ij}^{sol}(x,t;\tilde{\mathcal{D}}(I))$ defined in (\ref{pij1}) and (\ref{pij2}) respectively, as $t\to\infty$ with $(x,t)\in C(x_1,x_2,v_1,v_2)$, their error is exponentially small with
\begin{align}
	p_{ij}^{sol}(x,t;\tilde{\mathcal{D}})=p_{ij}^{sol}(x,t;\tilde{\mathcal{D}}(I))+\mathcal{O}(e^{-a\mu(I)t}).
\end{align}
\end{corollary}

\section{The small norm RH problem  for error function }\label{sec7}

\quad In this section,  we consider the error matrix-function $E(z)$ and  show that  the error function $E(z)$ solves a small norm Riemann-Hilbert problem which  can be expanded asymptotically for large times.
From the definition (\ref{transMr}), we can obtain a RH problem  for the matrix function  $E(z)$.

\begin{RHP}\label{RHP7}
	Find a matrix-valued function $E(z)$  with following identities:
	
	$\blacktriangleright$ Analyticity: $E(z)$ is analytical  in $\mathbb{C}\setminus  \Sigma^{(2)} $;

	$\blacktriangleright$ Asymptotic behaviors:
	\begin{align}
		&E(z) \sim I+\mathcal{O}(z^{-1}),\hspace{0.5cm}|z| \rightarrow \infty;
	\end{align}

	$\blacktriangleright$ Jump condition: $E$ has continuous boundary values $E_\pm$ on $\Sigma^{(2)}$ satisfying
	$$E_+(z)=E_-(z)V^{E},$$
	where the jump matrix $V^{E}$ is given by
	\begin{equation}
		V^{E}(z)=M^{sol}(z)V^{(2)}(z)M^{sol}(z)^{-1}. \label{VE}
	\end{equation}
\end{RHP}

{Proposition \ref{unim}} implies that $M^{sol}(z)$ is bound on $\Sigma^{(2)}$. By using proposition \ref{pro3v2} and Corollary \ref{v2p}, we have the following evaluation
\begin{equation}
\parallel V^{E}-I \parallel_p\lesssim \parallel V^{(2)}-I \parallel_p=\mathcal{O}(e^{- \varrho\rho_0t} ) ,\hspace{0.3cm}\text{for $1\leq p \leq +\infty$.} \label{VE-I}
\end{equation}
This uniformly vanishing bound $\parallel V^{E}-I \parallel$ establishes RHP \ref{RHP7} as a small-norm Riemann-Hilbert problem.
Therefore,    the   existence and uniqueness  of  the RHP \ref{RHP7} is  shown  by using  a  small-norm RH problem \cite{RN9,RN10} with
\begin{equation}
E(z)=I+\frac{1}{2\pi i}\int_{\Sigma^{(2)}}\dfrac{\left( I+\eta(s)\right) (V^E-I)}{s-z}ds,\label{Ez}
\end{equation}
where the $\eta\in L^2(\Sigma^{(2)})$ is the unique solution of following equation:
\begin{equation}
(1-C_E)\eta=C_E\left(I \right).
\end{equation}
Here $C_E$:$L^2(\Sigma^{(2)})\to L^2(\Sigma^{(2)})$ is a integral operator defined by
\begin{equation}
C_E(f)(z)=C_-\left( f(V^E-I)\right) ,
\end{equation}
with  the Cauchy projection operator $C_-$    on $\Sigma^{(2)}$ :
\begin{equation}
C_-(f)(s)=\lim_{z\to \Sigma^{(2)}_-}\frac{1}{2\pi i}\int_{\Sigma^{(2)}}\dfrac{f(s)}{s-z}ds.
\end{equation}
Then by (\ref{VE}) we have
\begin{equation}
\parallel C_E\parallel\leq\parallel C_-\parallel \parallel V^E-I\parallel_\infty \lesssim \mathcal{O}(e^{- \varrho\rho_0t} ),
\end{equation}
which means $\parallel C_E\parallel<1$ for sufficiently large t,   therefore  $1-C_E$ is invertible,  and   $\eta$  exists and is unique.
Moreover,
\begin{equation}
\parallel \eta\parallel_{L^2(\Sigma^{(2)})}\lesssim\dfrac{\parallel C_E\parallel}{1-\parallel C_E\parallel}\lesssim\mathcal{O}(e^{- \varrho\rho_0t} ).\label{normeta}
\end{equation}
Then we have the existence and boundedness of $E(z)$. In order to reconstruct the solution $p_{ij}(x,t)$ of (\ref{3w}), we need the asymptotic behavior of $E(z)$ as $z\to \infty$ .
\begin{Proposition}\label{asyE}
	For $E(z)$ defined in (\ref{Ez}), as $z\to \infty$, $E(z)$ has Laurent expansion  as
	\begin{align}
	E(z)=I+\frac{E_1}{z}+\mathcal{O}(z^{-2}),\label{expE}
	\end{align}
	where
	\begin{equation}
	E_1=-\frac{1}{2\pi i}\int_{\Sigma^{(2)}}\left( I+\eta(s)\right) (V^{E}-I)ds.
	\end{equation}
	Moreover,   $E_1$ satisfies following long time asymptotic behavior condition:
	\begin{equation}
	|E_1|\lesssim\mathcal{O}(e^{- \varrho\rho_0t}) ,\hspace{0.5cm}E_1\lesssim\mathcal{O}(e^{- 2\rho_0t}).\label{E1t}
	\end{equation}
\end{Proposition}
\begin{proof}
	By combining (\ref{normeta}) and (\ref{VE-I}), we obtain the  result promptly.
\end{proof}

\section{ Asymptotic analysis on the pure $\bar{\partial}$-Problem}\label{sec8}
\quad Now we consider  the long time asymptotics behavior of $M^{(3)}$.
The $\bar{\partial}$-problem 4  of $M^{(3)}$ is equivalent to the integral equation
\begin{equation}
M^{(3)}(z)=I-\frac{1}{\pi}\iint_\mathbb{C}\dfrac{M^{(3)}(s)W^{(3)} (s)}{s-z}dm(s),\label{m3}
\end{equation}
where $m(s)$ is the Lebesgue measure on the $\mathbb{C}$. Denote $C_z$ as the left Cauchy-Green integral  operator,
\begin{equation*}
fC_z(z)=\frac{1}{\pi}\iint_C\dfrac{f(s)W^{(3)} (s)}{z-s}dm(s).
\end{equation*}
Then above equation can be rewritten as
\begin{equation}
M^{(3)}(z)=I\cdot\left(I-C_z \right) ^{-1}.\label{deM3}
\end{equation}
To prove the existence of operator $\left(I-C_z \right) ^{-1}$, we have following Lemma.
\begin{lemma}\label{Cz}
	The norm of the integral operator $C_z$ decay to zero as $t\to\infty$:
	\begin{equation}
	\parallel C_z\parallel_{L^\infty\to L^\infty}\lesssim |t|^{-1/2},
	\end{equation}
	which implies that  $\left(I-C_z \right) ^{-1}$ exists.
\end{lemma}
\begin{proof}
	For any $f\in L^\infty$,
	\begin{align}
	\parallel fC_z \parallel_{L^\infty}&\leq \parallel f \parallel_{L^\infty}\frac{1}{\pi}\iint_C\dfrac{|W^{(3)} (s)|}{|z-s|}dm(s)\nonumber.
	\end{align}
	Consequently, we only need to  evaluate the integral
	$\frac{1}{\pi}\iint_C\dfrac{|W^{(3)} (s)|}{|z-s|}dm(s)$. As $W^{(3)} (s)$ is a sectorial function,  we only need to consider it on ever sector. Recall the definition of $W^{(3)} (s)=M^{r}(z)\bar{\partial}R^{(2)}(z)M^{r}(z)^{-1}$.  $W^{(3)} (s)\equiv0$ out of $\bar{\Omega}$. We only detail the case for matrix functions having support in the sector $\Omega_1$ as $-n_{13}<\xi<-n_{12}$, because the case $\xi<-n_{12}$ is more trivial. Then in this case, $\theta_{12}<0$.
	Proposition \ref{asyE} and \ref{unim}  implies  the boundedness of $M^{r}(z)$ and $M^{r}(z)^{-1}$ for $z\in \bar{\Omega}$, so
	\begin{equation}
		\frac{1}{\pi}\iint_{\Omega_1}\dfrac{|W^{(3)} (s)|}{|z-s|}dm(s)\lesssim \iint_{\Omega_1}\dfrac{|\bar{\partial}R_{11} (s) e^{-izt\theta_{12}}|+|\bar{\partial}R_{12} (s) e^{izt\theta_{23}}|+|\bar{\partial}R_{13} (s) e^{izt\theta_{13}}|}{|z-s|}dm(s).
	\end{equation}
Give the detail of  estimation to first integral, and the others are  similarly. Referring  to (\ref{dbarRj}) in proposition \ref{proR}, the integral $\iint_{\Omega_1}\dfrac{|\bar{\partial}R_1 (s)e^{-izt\theta_{12}}|}{|z-s|}dm(s)$	 can be divided to five part:
\begin{align}
	\iint_{\Omega_1}\dfrac{|\bar{\partial}R_1 (s)e^{-izt\theta_{12}}|}{|z-s|}dm(s)\lesssim I_{11}+I_{12}+I_{21}+I_{22}+I_3,	
\end{align}
with
\begin{align}
	&I_{11}=\iint_{\Omega_{10}}\dfrac{|q_{11}' (\text{Re}s)e^{-izt\theta_{12}}|}{|z-s|}dm(s),\hspace{0.5cm}I_{12}=\iint_{\Omega_{10}}\dfrac{|q_{11} (\text{Re}s)e^{-izt\theta_{12}}|}{|z-s|}dm(s)\\
	&I_{21}=\iint_{\Omega_{11}}\dfrac{|q_{11}' (|s-i\varrho|)e^{-izt\theta_{12}}|}{|z-s|}dm(s),\hspace{0.5cm}I_{22}=\iint_{\Omega_{11}}\dfrac{|\mathcal{X}_{2}' (|s-i\varrho|)e^{-izt\theta_{12}}|}{|z-s|}dm(s)\\
	&I_3=\iint_{\Omega_{11}}\dfrac{|s-i\varrho|^{-1/2}|e^{-izt\theta_{12}}|}{|z-s|}dm(s).	
\end{align}
For $I_{11}$ and $I_{12}$, let $s=u+vi$, $z=x+yi$. In the following computation,  we will use the inequality
\begin{align}
	\parallel |s-z|^{-1}\parallel_{L^q(0,+\infty)}&=\left\lbrace \int_{0}^{+\infty}\left[  \left( \frac{u-x}{v-y}\right) ^2+1\right]^{-q/2}
d\left( \frac{u-x}{|v-y|}\right)\right\rbrace ^{1/q}|v-y|^{1/q-1}\nonumber\\
&\lesssim |v-y|^{1/q-1},
\end{align}
with $1\leq q<+\infty$ and $\frac{1}{p}+\frac{1}{q}=1$. Moreover, by $-n_{13}<\xi<-n_{12}$ we have the constant $\theta_{12}<0$.
Therefore,
	\begin{align}
	I_{11}&\leq \int_{0}^{\varrho}\int_{0}^{+\infty}\dfrac{|q_{11}' (u)|}{|z-s|}e^{tv\theta_{12}}dudv\leq \int_{0}^{\varrho}\parallel |s-z|^{-1}\parallel_{L^2(\mathbb{R}^+)} \parallel q_{11}'\parallel_{L^2(\mathbb{R}^+)} e^{\theta_{12}tv}dv\nonumber\\
	&\lesssim \int_{0}^{+\infty}|v-y|^{-1/2} e^{tv\theta_{12}}dv\lesssim t^{-1/2}.
	\end{align}
It also deduce to that $I_{12}\lesssim t^{-1/2}$ in the same way. For the $\Omega_{10}$,  we take another change of variable as  $s=u+(v+\varrho)i$, $z=x+(y+\varrho)i$. Recall $\varphi$ which is the  angle  of $\Sigma_{1}$, then
\begin{align}
	I_{21}&\leq \int_{0}^{+\infty}\int_{\frac{v}{\tan\varphi}}^{+\infty}\dfrac{|q_{11}' (\sqrt{u^2+v^2})|}{|z-s|}e^{tv\theta_{12}}dudve^{t\varrho\theta_{12}}\nonumber\\
	&\leq \int_{0}^{+\infty}\parallel |s-z|^{-1}\parallel_{L^2(\mathbb{R}^+)} \parallel q_{11}'\parallel_{L^2(\mathbb{R}^+)}e^{tv\theta_{12}}dve^{t\varrho\theta_{12}}\nonumber\\
	&\lesssim \int_{0}^{+\infty}|v-y|^{-1/2} e^{tv\theta_{12}}dve^{t\varrho\theta_{12}}\lesssim t^{-1/2}e^{t\varrho\theta_{12}}.
\end{align}
In the second  inequality we use that d$u=\sqrt{1+\left( \frac{v}{u}\right) ^2}$d$\sqrt{u^2+v^2}\lesssim$d$\sqrt{u^2+v^2}$. This evaluation is also practicable for $I_{22}$ because $\mathcal{X}_{2}'\in L^2(\mathbb{R})$.
Before we estimating the last item, we  consider for $p>2$,
\begin{align}
	\left( \int_{v}^{+\infty}|\sqrt{u^2+v^2}|^{-\frac{p}{2}}du\right) ^{\frac{1}{p}}&=\left( \int_{v}^{+\infty}|l|^{-\frac{p}{2}+1}u^{-1}dl\right) ^{\frac{1}{p}}\lesssim v^{-\frac{1}{2}+\frac{1}{p}}.
\end{align}
By Cauchy-Schwarz inequality,
\begin{align}
	I_2&\leq \int_{0}^{+\infty}\parallel |s-z|^{-1}\parallel_{L^q(\mathbb{R}^+)} \parallel |z-i\varrho|^{-1/2}\parallel_{L^p(\mathbb{R}^+)}e^{tv\theta_{12}}dve^{t\varrho\theta_{12}}\nonumber\\
	&\lesssim\int_{0}^{+\infty}|v-y|^{1/q-1}v^{-\frac{1}{2}+\frac{1}{p}}e^{tv\theta_{12}}dve^{t\varrho\theta_{12}}\nonumber\\
	&\lesssim\int_{0}^{+\infty}v^{-\frac{1}{2}}e^{tv\theta_{12}}dve^{t\varrho\theta_{12}}\lesssim t^{-1/2}e^{t\varrho\theta_{12}}.
\end{align}
  So the proof is completed.
\end{proof}

To reconstruct the solution of (\ref{3w}), we need following proposition.
\begin{Proposition}\label{asyM3i}
	As $z\to \infty$, The solution $M^{(3)}(z)$  of  $\bar{\partial}$-problem  admits Laurent expansion:
	\begin{equation}
		M^{(3)}(z)=I+\frac{1}{z}M^{(3)}_1(x,t)+\mathcal{O}(z^{-2}),
	\end{equation}
	where $M^{(3)}_1$ is a $z$-independent coefficient with
	\begin{equation}
		M^{(3)}_1(x,t)=-\frac{1}{\pi}\iint_CM^{(3)}(s)W^{(3)} (s)dm(s).
	\end{equation}
	There exist constants $T_1$, such that for all $t>T_1$, $M^{(3)}_1(x,t)$  satisfies
	\begin{equation}
		|M^{(3)}_1(x,t)|\lesssim t^{-1}.\label{M31}
	\end{equation}
\end{Proposition}
\begin{proof}
	The proof proceeds along the same lines as it of above Proposition. 	{Lemma \ref{Cz}} and (\ref{deM3}) implies that for large $t$,   $\parallel M^{(3)}\parallel_\infty \lesssim1$. And for same reason, we only estimate the integral on sector $\Omega_1$ as $-n_{13}<\xi<-n_{12}$. Referring  to (\ref{dbarRj}) in proposition \ref{proR}, the region $\Omega_1$ as a  domain of integration is divided into three parts shown in Figure \ref{FigureO}. In this Figure, $\Omega_{12}=\Omega_{10}\cap\mathbb{D}(i\varrho,2)$, and $\Omega_{13}=\Omega_{10}\setminus\Omega_{12}$ with $\Omega_{10}$ shown in Figure \ref{figR2}.
	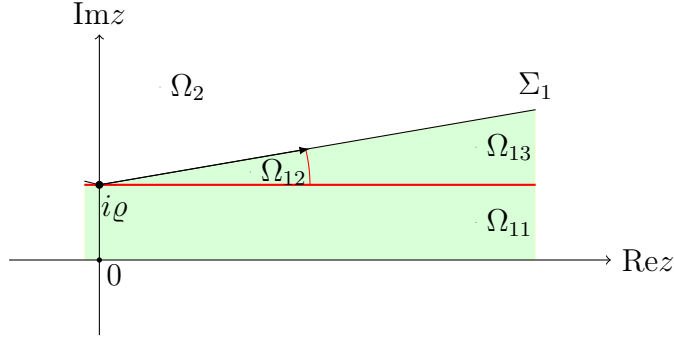
\begin{figure}[h]
		\centering
		\begin{tikzpicture}[node distance=2cm]
			\draw[yellow!20, fill=green!15] (-2.8,1)--(3,2)--(3,0)--(-3,0)--(-3,1)--(-2.8,1);
			\draw(-2.8,1)--(3,2)node[above]{$\Sigma_1$};
			\draw[red] (0,1) arc (0:14:2);
			\draw[-](-2.8,1)--(-3,1.05);
			\draw[->](-4,0)--(4,0)node[right]{ Re$z$};
			\draw[->](-2.8,-1)--(-2.8,3)node[above]{ Im$z$};
			\draw[-][thick, red](-3,1)--(3,1);
			\draw[-latex](-2.8,1)--(0,1.48);
			\coordinate (I) at (-2.8,0);
			\coordinate (C) at (-0.2,2.2);
			\coordinate (D) at (2.2,0.5);
			\fill (D) circle (0pt) node[right] {$\Omega_{11}$};
			\coordinate (J) at (-2.2,-0.2);
			\fill (I) circle (1pt);
              \node  at (-2.6,-0.2)  {$0$};
			\coordinate (a) at (2.2,1.5);
			\fill (a) circle (0pt) node[right] {$\Omega_{13}$};
			\coordinate (c) at (-0.8,1.17);
			\fill (c) circle (0pt) node[right] {$\Omega_{12}$};
			\coordinate (b) at (-2,2.3);
			\fill (b) circle (0pt) node[right] {$\Omega_{2}$};
			\coordinate (A) at (-2.8,1);
			\fill (A) circle (1.5pt) node[right] {};
			\coordinate (B) at (-2.6,1);
			\fill (B) circle (0pt) node[below] {$i\varrho$};
		\end{tikzpicture}
		\caption{The cyan region is the integral domain $\Omega_1$, which is divided into three parts as $\Omega_{1i}$, $i=1,2,3$. The red  arc is $\Omega_{10}\cap\partial\mathbb{D}(i\varrho,2)$, which divides $\Omega_{10}$ in figure \ref{figR2} into two parts:  $\Omega_{12}=\Omega_{10}\cap\mathbb{D}(i\varrho,2)$, and $\Omega_{13}=\Omega_{10}\setminus\Omega_{12}$. }
		\label{FigureO}
	\end{figure}
 	
 	We also divide $M^{(3)}_1$ to six parts, but  this time we use another estimation (\ref{dbarRj3}):
	\begin{equation}
	|\frac{1}{\pi}\iint_CM^{(3)}(s)W^{(3)} (s)dm(s)|\lesssim I_{41}+I_{42}+I_{51}+I_{52}+I_{61}+I_{62},
	\end{equation}
	with
	\begin{align}
		&I_{41}=\iint_{\Omega_{11}}|q_{11}' (\text{Re}s)|e^{t\text{Im}s\theta_{12}}dm(s),\hspace{0.5cm}
		I_{42}=\iint_{\Omega_{11}}|q_{11} (\text{Re}s)|e^{t\text{Im}s\theta_{12}}dm(s),\nonumber\\
		&I_{51}=\iint_{\Omega_{10}}|q_{11}' (|s-i\varrho|)|e^{t\text{Im}s\theta_{12}}dm(s),\hspace{0.5cm}I_{52}=\iint_{\Omega_{10}}|\mathcal{X}_{2}' (|s-i\varrho|)|e^{t\text{Im}s\theta_{12}}dm(s)\\
		&I_{61}=\iint_{\Omega_{12}}\frac{|\mathcal{X}_2(|z-i\varrho|)||q_{11}(|z-i\varrho|)-q_{11}(0)|}{|s-i\varrho|}e^{t\text{Im}s\theta_{12}}dm(s),\\
		&I_{62}=\iint_{\Omega_{13}}\frac{|\mathcal{X}_2(|z-i\varrho|)||q_{11}(|z-i\varrho|)-q_{11}(0)|}{|s-i\varrho|}e^{t\text{Im}s\theta_{12}}dm(s).	
	\end{align}	
	Because $r\in H^{1,1}(\mathbb{R})$, $r'$ and $r$ are in $ L^1(\mathbb{R})$, which also implies $q_{11}'$ and $q_{11}$ in $ L^1(\mathbb{R})$. For $I_{41}$ and $I_{42}$, let $s=u+vi$
	\begin{align}
		I_{41}\leq&\int_{0}^{+\infty}\int_{0}^{\varrho}|q_{11}' (u)|e^{tv\theta_{12}}dudv\nonumber\\
		&\leq \int_{0}^{\varrho} \parallel q_{11}'\parallel_{L^1(\mathbb{R}^+)}  e^{tv\theta_{12}} dv\lesssim \int_{0}^{\varrho}  e^{tv\theta_{12}} dv\leq t^{-1}.\nonumber
	\end{align}
$I_{42}$ has same evaluation. For the rest of integral,  we take another change of variable as  $s=u+(v+\varrho)i$ with $|s-i\varrho|=l$. Then Recall   $du\lesssim dl$ in $\Omega_{11}$,
\begin{align}
	&I_{51}=\int_{0}^{+\infty}\int_{\frac{v}{\tan\varphi}}^{+\infty}|q_{11}' (|s-i\varrho|)|e^{t\text{Im}s\theta_{12}}dudve^{t\varrho\theta_{12}}\nonumber\\
	&\lesssim\int_{0}^{+\infty}\parallel q_{11}'\parallel_{L^2(\mathbb{R}^+)}e^{tv\theta_{12}}dve^{t\varrho\theta_{12}}\lesssim t^{-1}e^{t\varrho\theta_{12}}.\nonumber
\end{align}
Similarly, $I_{52}\lesssim t^{-1}e^{t\varrho\theta_{12}}$ from $\mathcal{X}_2'\in L^2(\mathbb{R})$. Recall the definition of $\mathcal{X}_2$ in (\ref{X2}). Then from (\ref{r-}), $I_{61}$ has
\begin{align}
	I_{61}&\leq\int_{0}^{2\tan\varphi}\int_{\frac{v}{\sin\varphi}}^{\frac{2}{\cos\varphi}}|\mathcal{X}_2(l)|l^{-1/2}e^{tv\theta_{12}}dldve^{tv\varrho\theta_{12}}\nonumber\\
	&\lesssim\parallel \mathcal{X}_2(l)l^{-1/2}\parallel _{L^1(0,2)}\int_{0}^{2\tan\varphi}v^{-1/2}e^{tv\theta_{12}}dve^{tv\varrho\theta_{12}}\lesssim t^{-1/2}e^{tv\varrho\theta_{12}}.
\end{align}
 And
\begin{align}
		I_{62}&=\iint_{\Omega_{13}}\frac{|q_{11}(|z-i\varrho|)-q_{11}(0)|}{|s-i\varrho|^2}e^{t\text{Im}s\theta_{12}}dm(s)\nonumber\\
	&\leq \int_{0}^{+\infty}\int_{2}^{+\infty}l^{-2}dle^{tv\theta_{12}}dve^{tv\varrho\theta_{12}}\lesssim t^{-1}e^{tv\varrho\theta_{12}}.\nonumber
\end{align}
Combining above  inequality we obtain the result.
\end{proof}

\section{Asymptotic stability of  N-soliton solutions  }\label{sec9}

\quad Now we begin to construct the long time asymptotics of the tree-wave equation (\ref{3w}).
 Recalling  a  serial  of transformations (\ref{transm1}), (\ref{transm2}), (\ref{transm3}) and (\ref{transMr}), we have
\begin{align}
M(z)=&M^{(3)}(z)E(z)M^{sol}(z)R^{(2)}(z)^{-1}T(z)^{-\sigma_3}.
\end{align}
To  reconstruct the solution $q_{ij}(x,t)$ for $i,j=1,2,3$, by using (\ref{recons u}),   we take $z\to \infty$ out of $\bar{\Omega}$. In this case,  $ R^{(2)}(z)=I$. Further using    {Propositions} \ref{proT},  \ref{asyE}  and  \ref{asyM3i},   we can obtain that as $z\to \infty$  behavior
\begin{align}
	M(z)= \left(I+ M^{(3)}_1 z^{-1}+\cdots \right) \left( I+E_1z^{-1}+\cdots \right) ( I+ M^{sol}_1 z^{-1} +\cdots )  \left( I- iT_1^{-\sigma_3}z^{-1}+\cdots \right),\nonumber
\end{align}
from which we obtain that
\begin{align}
	M(z)= I+ (  M^{sol}_1 -iT_1^{-\sigma_3} + \mathcal{O}(t^{-1}) ) z^{-1} + \mathcal{O}(z^{-2}).
\end{align}
Substituting  above estimation into  (\ref{recons u}) leads to
\begin{align}
	p_{ij}(x,t)&=-i(a_i-a_j)\lim_{z\to \infty}[zM]_{ij}
 =p_{ij}^{sol}(x,t;\widetilde{\mathcal{D}}(I) ) +\mathcal{O}(t^{-1}),  \label{resultu}
\end{align}
where $p_{ij}^{sol}(x,t;\widetilde{\mathcal{D}}(I))$ is shown in Corollary \ref{sol}.
Therefore, we achieve main result of this paper.

\begin{theorem}\label{last}   Let $p_{ij}(x,t)$ be the solution for  the initial-value problem (\ref{3w}) with generic data   $p_{ij,0}(x)\in H^{1,2}(\mathbb{R})\cap P_0$ and scatting data $\left\lbrace  \vec{r}(z)=(r_1(z),...,r_4(z)),\left\lbrace z_n,c_n\right\rbrace^{4N_1+2N_2}_{n=1}\right\rbrace$.
Denote $p_{ij}^{sol}(x,t;\widetilde{\mathcal{D}}(I) )$ be the $\mathcal{N}(I)$-soliton solution corresponding to   scattering data
$\widetilde{\mathcal{D}}(I) =\left\lbrace  0,\left\lbrace z_n,\widetilde{c}_n(I)\right\rbrace_{n\in\mathcal{Z}(I)}\right\rbrace$ shown in Corollary \ref{sol}.
There exist a large constant $t_0=t_0(\xi)$, for all $ t> t_0$,
\begin{align}
		p_{ij}(x,t)=p_{ij}^{sol}(x,t;\widetilde{\mathcal{D}}(I)) +\mathcal{O}(t^{-1}).
\end{align}

\end{theorem}

\noindent\textbf{Acknowledgements}

This work is supported by  the National Natural  Science
Foundation of China (Grant No. 11671095,  51879045).

\hspace*{\parindent}
\\


\begin{thebibliography}{10}


\bibitem{BUCK}
R. J. Buckingham, R. M. Jenkins, P. D. Miller
\newblock  Semiclassical Soliton Ensembles for the Three-Wave Resonant Interaction Equations,
\newblock {\em  Comm. Math. Phys.},  354(2017), 1015-1100.

\bibitem{33}
C. Lange and A. Newell,
\newblock {Spherical shells like hexagons: cylinders prefer diamonds, }
\newblock {\em  J. Appl.	Mech.}, 40(1973), 575-581.

\bibitem{42}
L. McGoldrick,
\newblock {Resonant interactions among capillary-gravity waves},
\newblock {\em  J. Fluid Mech}, 21(1965), 305-331.

\bibitem{49}
R. Sagdeev and A. Galeev,
\newblock   Nonlinear Plasma Theory, 
\newblock {\em Frontiers in Physics},   , W. A. Benjamin, New
York, NY, (34)1969.

\bibitem{51}
L. Stenflo,
\newblock   Resonant three-wave interactions in plasmas,
\newblock {\em Phys. Scr},  T50(1994), 15-19.

\bibitem{55}
A. Newell,
\newblock   Rossby wave packet interactions,
\newblock {\em J. Fluid Mech.},  35(1969), 255-271.

\bibitem{2}
 F. Baronio, M. Conforti, M. Andreana, V. Couderc, C. De Angelis, S. Wabnitz, A. Barth\'el\'emy, and
A. Degasperis,
\newblock   Frequency generation and solitonic decay in three-wave interactions,
\newblock {\em Opt. Express},  17(2009), 13889-13894.

\bibitem{39}
W. Mak, B. Malomed, and P. Chu,
\newblock    Three-wave gap solitons in wave guides with quadratic nonlinearity,
\newblock {\em Phys. Rev. E},  58(1998),  6708-6722.





\bibitem{lax1}
Y. S. Li,
\newblock  Soliton and Integrable System,
\newblock {\em   Advanced Series in Nonlinear Science, }  1999 (in Chinese).

\bibitem{lax2}
V. E. Zakharov and S. V. Manakov,
\newblock  The theory of resonance interaction of wave packets in nonlinear media,
\newblock {\em   Zh. Eksp. Teor. Fiz, }   69(1975), 1654-1673
(in Russian).


\bibitem{D2011}
A. Degasperis, M. Conforti  and F. Baronio et al,
\newblock  The three-wave resonant interaction equations: spectral and numerical methods,
\newblock {\em  Lett. Math. Phy.},   96(2011), 367-403.


\bibitem{Du1}
B. A. Dubrovin,
\newblock   Theta-functions and nonlinear equations,
\newblock {\em  Russian Math. Surveys, } 36 (1981), 11-92.


\bibitem{Du2}
A. Reiman,
\newblock  Space-time evolution of nonlinear three-wave interactions. II. Interaction in an
inhomogeneous medium,
\newblock {\em    Rev. Mod. Phys., } 51 (1979),   311-330.


\bibitem{xu}
J. Xu and E. G. Fan,
\newblock  The three-wave equation on the half-line,
\newblock {\em  Phy. Lett. A, }  378(2014) 26-33.


\bibitem{HG}
G. L. He,  X. G.  Geng,  L. H.  Wu,
\newblock  Algebro-geometric quasi-periodic solutions to the three-wave
resonant interaction hierarchy,
\newblock {\em  SIAM  J. Math. Anal., } 46(2014) 1348-1384.



\bibitem{WG}
Y. T.   Wu,    X. G.  Geng,  L. H.  Wu,
\newblock  A finite-dimensional integrable system associated with the three-wave interaction equations,
\newblock {\em     J. Math. Phys., } 40(1999)  3409-3430 .



\bibitem{30}
D. Kaup, A. Reiman, and A. Bers,
\newblock  Space-time evolution of nonlinear three-wave interactions: I. Interaction in a homogeneous medium,
\newblock {\em   Rev. Mod. Phys., }  51, 275-310.

\bibitem{41}
R. Martin and H. Segur,
\newblock  Toward a general solution of the three-wave partial differential equations,
\newblock {\em   Stud. Appl. Math., }  137(2016), 70-92.


\bibitem{global}
J. Rauch,
\newblock  Hyperbolic Partial Differential Equations and Geometric Optics,
\newblock {\em   Graduate Studies in Mathematics 133, Amer. Math. Soc.}, Providence, RI, 2012.




\bibitem{Gardner1967}  C. S. Gardner,  J. M. Green,  M. D.  Kruskal  and R. M.  Miura, \newblock  Method
for solving the Korteweg-de Vries equation, \newblock {\em Phys. Rev. Lett.}, 19(1967), 1095-1097.



\bibitem{Zakharov1974} V. E. Zakharov  and A. B.  Shabat,
\newblock A scheme for integrating the nonlinear
equations of mathematical physics by the method of the inverse scattering
problem, \newblock {\em Funk. Anal. Pril.}, 6(1974),   43-53.



\bibitem{Zakharov1979}  V. E. Zakharov  and A. B. Shabat, \newblock  A scheme for integrating the nonlinear
equations of mathematical physics by the method of the inverse scattering
problem. II, \newblock {\em Funk. Anal. Pril.}, 13(1979),   13-22.


\bibitem{Beals1981} R. Beals  and R. R. Coifman, \newblock Scattering, transformations spectrales
et equations d'evolution nonlineare. I. Seminaire Goulaouic-Meyer-Schwartz,
exp. 22 (1981).

\bibitem{Beals1982}  R.  Beals and R. R. Coifman, \newblock Scattering,   transformations spectrales et
equations d'evolution nonlineare. II. Seminaire Goulaouic-Meyer-Schwartz,
exp. 21(1982).


\bibitem{Manakov1974} S. V. Manakov, \newblock Nonlinear Fraunhofer diffraction,  \newblock {\em Sov. Phys.-JETP} 38(1974), 693-696.

\bibitem{ZM1976}
V. E. Zakharov, S. V.  Manakov,
\newblock {Asymptotic behavior of nonlinear wave systems integrated by the inverse scattering method},
\newblock {\em  Soviet Physics JETP,}   44(1976), 106-112.



\bibitem{SPC}
P. C. Schuur,
\newblock {Asymptotic analysis of soliton products},
\newblock {\em Lecture Notes in Mathematics,}  1232, 1986.


\bibitem{BRF}
R. F. Bikbaev,
\newblock {Asymptotic-behavior as t-infinity of the solution to the cauchy-problem for the Landau-Lifshitz equation},
\newblock {\em  Theor. Math. Phys,}  77(1988), 1117-1123.

\bibitem{Foka}
R. F. Bikbaev,
\newblock {Soliton generation for initial-boundary-value problems  },
\newblock {\em  Phys. Rev. Lett.}, 68(1992), 3117-3120.





\bibitem{RN6}
X. Zhou, P. Deift,
\newblock  A steepest descent method for oscillatory Riemann-Hilbert problems.
\newblock {\em Ann. Math.}, 137(1993),  295-368.


\bibitem{RN9}
X. Zhou, P. Deift,
\newblock  Long-time behavior of the non-focusing nonlinear Schr$\ddot{o}$dinger equation--a case study,
\newblock {\em Lectures in Mathematical Sciences}, Graduate School of Mathematical Sciences, University of Tokyo, 1994.


\bibitem{RN10}
P. Deift, X. Zhou,
\newblock Long-time asymptotics for solutions of the NLS equation with initial data in a weighted Sobolev space,
\newblock {\em Comm. Pure Appl. Math.}, 56(2003), 1029-1077.



\bibitem{Grunert2009}
K. Grunert,  G. Teschl,
\newblock   Long-time asymptotics for  the Korteweg de Vries equation  via  noninear  steepest descent.
\newblock {\em Math. Phys. Anal. Geom.},     12(2009), 287-324.

\bibitem{MonvelCH}
 A. Boutet de Monvel,  A. Kostenko, D. Shepelsky, G. Teschl,
\newblock  Long-time asymptotics for the Camassa-Holm equation.,
\newblock {\em SIAM J. Math. Anal}, 41(2009),  1559-1588.


\bibitem{Monvel1} A.  Boutet de Monvel, A. Its,  V.Kotlyarov, \newblock Long-time asymptotics for the focusing NLS equation with
time-periodic boundary condition on the half-line. \newblock  {\em Comm. Math. Phys.}, 290(2009), 479-522.


\bibitem{Monvel2} A. Boutet de Monvel, J. Lenells,  D. Shepelsky, \newblock Long-time asymptotics for the Degasperis-Procesi
equation on the half-line. \newblock  {\em Ann. Inst. Fourier}, 69(2019), 171-230.

\bibitem{xu2015}
J. Xu, E. G. Fan,
\newblock  Long-time asymptotics for the Fokas-Lenells equation with decaying initial value problem: Without solitons,
\newblock {\em  J. Differ. Equ.,}    259(2015), 1098-1148.

\bibitem{xusp}
J. Xu,
\newblock  Long-time asymptotics for the short pulse equation,
\newblock {\em  J. Differ. Equ.,}  265(2018), 3494-3532.


\bibitem{Geng3}  H. Liu,  X.G. Geng,  B.  Xue, \newblock   The Deift-Zhou steepest descent method to long-time asymptotics for the
Sasa-Satsuma equation. \newblock  {\em J. Differ. Equ.}, 265(2018), 5984-6008.

\bibitem{XF2020} J. Xu, E. G. Fan, \newblock   Long-time asymptotic behavior for the complex short pulse equation,
 \newblock  {\em   J. Differ. Equ.},  269(2020), 10322-10349.



\bibitem{MandM2006}
K. T. R. McLaughlin, P. D. Miller,
\newblock {The $\bar{\partial}$ steepest descent method and the asymptotic behavior of polynomials orthogonal on
	the unit circle with fixed and exponentially varying non-analytic weights},
\newblock {\em Int.  Math. Res. Not.}, (2006), Art. ID 48673.

\bibitem{MandM2008}
K. T. R. McLaughlin, P. D. Miller,
\newblock {The $\bar{\partial}$ steepest descent method for orthogonal polynomials on the real line with varying weights},
\newblock {\em  Int. Math. Res. Not.},  (2008), Art. ID   075.

\bibitem{DandMNLS}
M. Dieng, K. D. T. McLaughlin,
\newblock {Dispersive asymptotics for linear and integrable equations by the Dbar steepest descent method},
\newblock {\em } Nonlinear dispersive partial differential equations and inverse scattering,
253-291, Fields Inst. Commun., 83, Springer, New York,  2019.

\bibitem{fNLS}
M. Borghese, R. Jenkins, K. T. R. McLaughlin, Miller P,
\newblock { Long-time asymptotic behavior of the focusing nonlinear Schr$\ddot{o}$dinger equation, }
\newblock {\em  Ann. I. H. Poincar$\acute{e}$ Anal}, 35(2018), 887-920.



\bibitem{Liu3}
R. Jenkins, J. Liu, P. Perry, C. Sulem,
\newblock   Soliton resolution for the derivative nonlinear Schr\"odinger equation,
\newblock {\em Comm. Math. Phys.},  363(2018), 1003-1049.

\bibitem{SandRNLS}
S. Cuccagna, R. Jenkins,
\newblock {On asymptotic stability of N-solitons of the defocusing nonlinear Schr$\ddot{o}$dinger equation, }
\newblock {\em  Comm. Math. Phys}, 343(2016), 921-969.



\bibitem{YF1} Y. L. Yang, E. G. Fan, \newblock { Soliton resolution for the short-pulse equation}, \newblock {\em J. Differ. Equ.}, 280(2021), 644-689.



\bibitem{YF3}  Y. L. Yang, E. G. Fan, \newblock Long-time asymptotic behavior of the modified Camassa-Holm equation, arXiv:2101.02489.

\bibitem{YF4}  Q. Y. Cheng, E. G. Fan,  \newblock Soliton resolution for the focusing Fokas-Lenells equation with weighted Sobolev initial data, arXiv:2010.08714.



\bibitem{Deift1982}  P. Deift, C. Tome,   E. Trubowitz, \newblock Inverse Scattering and the Boussinesq Equation, \newblock {\em  Comm.  Pure   Appl.  Math.},  35(1982), 567-628.

\bibitem{Constantin1}
A. Constantin, R. I. Ivanov, J. Lenells,
\newblock Inverse scattering transform for the Degasperis-Procesi equation.\newblock {\em
Nonlinearity}, 23(2010), 2559-2575.


\bibitem{Monvel3}   A. Boutet de Monvel,  A.  Shepelsky, \newblock  A Riemann-Hilbert approach for the Degasperis-Procesi equation.
\newblock {\em 	Nonlinearity}, 26(2013), 2081-2107.

\bibitem{Lenells1}  C. Charlier, J. Lenells,
\newblock The "good" Boussinesq equation: a Riemann-Hilbert approach, arXiv:2003.02777.

\bibitem{Lenells2}  C. Charlier, J. Lenells, D. Wang,  \newblock  The "good" Boussinesq equation: long-time asymptotics, arXiv:2003.04789.

\bibitem{Geng1}  X. G. Geng, H. Liu,
\newblock The nonlinear steepest descent method to long-time asymptotics of the coupled
nonlinear Schroinger equation. \newblock {\em J. Nonlinear Sci.}, 28(2018), 739-763.



\bibitem{Geng2}
X. G. Geng,  M. M. Chen, K. D. Wang,
\newblock  Long-time asymptotics of the coupled modified Korteweg de
Vries equation. \newblock {\em  J. Geom. Phys.}, 142(2019), 151-167.

\bibitem{BC} R.  Beals, R. R.   Coifman,  \newblock  Scattering and inverse scattering for first-order systems.\newblock {\em    Comm. Pure	Appl. Math.}, 37(1984), 39-90.



\end{thebibliography}
\end{document}